\documentclass[11pt]{article}

\usepackage{amsmath,amsfonts,amsthm,amssymb}
\usepackage[margin=1in,letterpaper]{geometry}
\usepackage{hyperref}
\usepackage{physics}
\usepackage{graphicx}
\usepackage{tikz}
\usepackage{tikz-cd}
\usetikzlibrary{arrows}
\usepackage{amsthm}
\theoremstyle{definition}
\newtheorem{defi}{Definition}[section]
\newtheorem{theorem}[defi]{Theorem}
\newtheorem{corollary}[defi]{Corollary}
\newtheorem{lemma}[defi]{Lemma}
\newtheorem{prop}[defi]{Proposition}

\newtheorem{remark}[defi]{Remark}

\usepackage{bm}
\usepackage{bbm}

\DeclareMathOperator{\Coker}{Coker}
\DeclareMathOperator{\spa}{span}
\DeclareMathOperator{\supp}{supp}

\DeclareMathOperator{\lk}{lk}
\DeclareMathOperator{\st}{st}

\DeclareMathOperator{\sd}{sd}
\DeclareMathOperator{\Soc}{Soc}

\providecommand{\keywords}[1]{\textbf{\textit{Keywords---}} #1}

\title{Multigraded strong Lefschetz property for balanced simplicial complexes}
\author{Ryoshun Oba\thanks{Department of Mathematical Informatics, Graduate School of Information Science and Technology, University of Tokyo, 7-3-1 Hongo, Bunkyo-ku, 113-8656, Tokyo Japan. Email: \texttt{ryoshun\_oba@mist.i.u-tokyo.ac.jp}}}

\begin{document}
\maketitle
\begin{abstract}
Generalizing the strong Lefschetz property for an $\mathbb{N}$-graded algebra, we introduce the multigraded strong Lefschetz property for an $\mathbb{N}^m$-graded algebra.
We show that, for $\bm{a} \in \mathbb{N}^m_+$, the generic $\mathbb{N}^m$-graded Artinian reduction of the Stanley-Reisner ring of an $\bm{a}$-balanced homology sphere over a field of characteristic $2$ satisfies the multigraded strong Lefschetz property.
A corollary is the inequality $h_{\bm{b}} \leq h_{\bm{c}}$ for $\bm{b} \leq \bm{c} \leq \bm{a}-\bm{b}$ among the flag $h$-numbers of an $\bm{a}$-balanced simplicial sphere. This can be seen as a common generalization of the unimodality of the $h$-vector of a simplicial sphere by Adiprasito and the balanced generalized lower bound inequality by Juhnke-Kubitzke and Murai.
We further generalize these results to $\bm{a}$-balanced homology manifolds and $\bm{a}$-balanced simplicial cycles over a field of characteristic $2$.
\end{abstract}
\keywords{Lefschetz property, Stanley-Reisner ring, balancedness, multigraded algebra, unimodality}

\section{Introduction}
The face numbers of simplicial complexes have been extensively studied in algebraic and topological combinatorics in the last decades.
A recent breakthrough announced by Adiprasito~\cite{Adi} (see also~\cite{APP,KX,PP}) is the hard Lefschetz theorem for the Stanley-Reisner ring of a simplicial (or homology) sphere, generalizing the work of Stanley~\cite{Sta80} for the boundary complex of simplicial polytopes.
Among many combinatorial consequences, this implies the generalized lower bound inequality (GLBI) asserting that the $h$-vector of a simplicial sphere is unimodal (more generally this implies the celebrated $g$-conjecture).
The balanced GLBI of Juhnke-Kubitzke and Murai~\cite{JM} (together with the hard Lefschetz theorem) asserts that the $h$-vector of a simplicial $(d-1)$-sphere satisfies the stronger inequality 
\begin{equation} \label{eq:balancedGLBI}
\frac{h_i}{\binom{d}{i}} \leq \frac{h_{i+1}}{\binom{d}{i+1}} \text{ for } i < \frac{d}{2}
\end{equation}
if its $1$-skeleton is $d$-colorable.
To bridge between these two, we investigate the face numbers of $\bm{a}$-balanced simplicial complexes.

For a positive integer vector $\bm{a}=(a_1,\ldots,a_m)$ with $|\bm{a}|:=a_1+\cdots+a_m=d$, a pair $(\Delta,\kappa)$ of a $(d-1)$-dimensional simplicial complex $\Delta$ and a vertex coloring $\kappa$ of $\Delta$ into $m$ colors is called \emph{$\bm{a}$-balanced} if each face of $\Delta$ contains at most $a_j$ vertices of color $j$ for each $j=1,\ldots,m$.
Stanley~\cite{Sta79} initiated the study of $\bm{a}$-balanced simplicial complexes in connection with fine multigraded algebra structure for its Stanley-Reisner ring.
In particular, the Stanley-Reisner ring of an $\bm{a}$-balanced simplicial complex with $\bm{a} \in \mathbb{N}^m$ admits a system of parameters homogeneous in the fine $\mathbb{N}^m$-grading induced by the coloring~\cite{Sta79}. 

With this in mind, we introduce multigraded strong Lefschetz property for an $\mathbb{N}^m$-graded algebra. (We defer most of the definitions until the following sections.)
Let $\mathbbm{k}$ be a field and let $A=\bigoplus_{\bm{0} \leq \bm{b} \leq \bm{a}} A_{\bm{b}}$ be an Artinian Gorenstein standard $\mathbb{N}^m$-graded $\mathbbm{k}$-algebra with $A_{\bm{0}}\cong A_{\bm{a}} \cong \mathbbm{k}$.
Here $\bm{c} \leq \bm{d}$ denotes the component-wise inequality $c_i \leq d_i$ for all $i$, and an $\mathbb{N}^m$-graded algebra is \emph{standard} if it is generated by degree $1$ elements under the coarse $\mathbb{N}$-grading.
We say that $A$ has the \emph{multigraded strong Lefschetz property} (as an $\mathbb{N}^m$-graded algebra) if there is an element $\ell_j \in A_{\bm{e}_j}$ for each $j=1,\ldots,m$ such that the multiplication map
\[
\times\bm{\ell}^{\bm{a}-2\bm{b}} : A_{\bm{b}} \rightarrow A_{\bm{a}-\bm{b}}
\]
is an isomorphism for all $\bm{b} \in \mathbb{N}^m$ with $\bm{b} \leq \frac{\bm{a}}{2}$. Here $\bm{e}_j \in \mathbb{N}^m$ is the $j$-th unit coordinate vector, and we denote $\bm{t}^{\bm{c}}=\prod_{j=1}^m t_j^{c_j}$ for $\bm{t}=(t_1,\ldots,t_m)$ and $\bm{c}=(c_1,\ldots,c_m)$. The elements $\ell_j$ ($j=1,\ldots,m$) are called the \emph{Lefschetz elements}.
We prove the following\footnote{As we were writing up this paper, we notice a connection of Theorem~\ref{thm:main multiHL} and bipartite rigidity~\cite{KNN} including an application to Gr\"{u}nbaum-Kalai-Sarkaria type inequality for $\bm{a}$-balanced simplicial complexes.
See an upcoming paper for details.}.
\begin{theorem} \label{thm:main multiHL}
Let $\mathbbm{k}$ be a field of characteristic $0$ or $2$ and let $(\Delta,\kappa)$ be an $\bm{a}$-balanced homology sphere\footnote{By abuse of words, we often say that $(\Delta,\kappa)$ satisfies a certain property is $\Delta$ satisfies it.} over $\mathbb{F}_2$. 
Then the generic $\mathbb{N}^m$-graded Artinian reduction $A=\widetilde{\mathbbm{k}}[\Delta]/(\Theta)$ of the Stanley-Reisner ring $\mathbbm{k}[\Delta]$ has the multigraded strong Lefschetz property.
\end{theorem}
\noindent
Here $\widetilde{\mathbbm{k}}$ is a purely transcendental field extension of $\mathbbm{k}$ obtained in the generic $\mathbb{N}^m$-graded Artinian reduction (see Section 4 for details).
Note that Theorem~\ref{thm:main multiHL} is a common generalization of the hard Lefschetz theorem of a simplicial sphere~\cite{Adi,PP,KX} and the dual weak Lefschetz property for a rank-selected subcomplex of a completely balanced simplicial sphere~\cite[Theorem 3.3]{JM} with an assumption on the characteristic of the field.
We conjecture Theorem~\ref{thm:main multiHL} holds for any infinite field $\mathbbm{k}$ and a homology sphere $\Delta$ over $\mathbbm{k}$.
Our proof of Theorem~\ref{thm:main multiHL} is based on an anisotropy technique over characteristic $2$ used in~\cite{APP,APPS,KX,PP}.
As a corollary of Theorem~\ref{thm:main multiHL}, we obtain the following combinatorial consequence for the flag $h$-vector $(h_{\bm{b}})_{\bm{b}}$ of an $\bm{a}$-balanced homology sphere over $\mathbb{F}_2$. See also Corollary~\ref{cor:h-vec} for a combinatorial corollary for the $h$-vector.
\begin{theorem} \label{thm:flag-h}
For an $\bm{a}$-balanced homology sphere $(\Delta,\kappa)$ over $\mathbb{F}_2$, we have $h_{\bm{b}} \leq h_{\bm{c}}$ for any $\bm{b},\bm{c} \in \mathbb{N}^m$ with $\bm{b} \leq \bm{c} \leq \bm{a}-\bm{b}$.
\end{theorem}
\noindent Note that Theorem~\ref{thm:flag-h} can be seen as a common generalization of GLBI and balanced GLBI. (GLBI is the case of $m=1$ in Theorem~\ref{thm:flag-h}. Balanced GLBI (\ref{eq:balancedGLBI}) follows from the inequality  $h_{i\bm{e}_j}\leq h_{(i+1)\bm{e}_j}$ for $\bm{a}=\bm{1}+2i\bm{e}_j$ together with the averaging argument of Goff, Klee and Novik~\cite{GKN}. See~\cite{Adi17,JM} for details.)

We further generalize the almost strong Lefschetz property of manifolds~\cite[Section 8]{AY} and the strong Lefschetz property of simplicial cycles (after Gorensteinification)~\cite[Theorem I]{APP} and the top-heavy strong Lefschetz property for doubly Cohen-Macaulay complexes~\cite[Corollary 3.2]{APP} to $\bm{a}$-balanced setting.
As a combinatorial corollary, we obtain the following generalization of Theorem~\ref{thm:flag-h} to manifolds (without boundary) concerning the flag $h''$-vector $(h''_{\bm{b}})_{\bm{b}}$. See also~Corollary~\ref{cor:mfd-h}. Note that over a field of characteristic $2$, every homology manifold is orientable.
\begin{theorem} \label{thm:manifold flag_h}
For $\bm{a} \in \mathbb{N}^m_+$, let $(\Delta,\kappa)$ be an $\bm{a}$-balanced connected homology manifold over $\mathbb{F}_2$.
Let $\bm{b}, \bm{c} \in \mathbb{N}^m$ be integer vectors with $\bm{b} \lneq \bm{c} \lneq \bm{a}-\bm{b}$\footnote{$\bm{b} \lneq \bm{c}$ stands for $\bm{b} \leq \bm{c}$ and $\bm{b} \neq \bm{c}$.}.
Then we have $h''_{\bm{c}} \geq h''_{\bm{b}}+ \binom{\bm{a}}{\bm{b}}\widetilde{\beta}_{|\bm{b}|}$, where $\widetilde{\beta}_i$ is the $i$-th reduced betti number of $\Delta$ over $\mathbb{F}_2$, and $\binom{\bm{a}}{\bm{b}}=\prod_{j=1}^m \binom{a_j}{b_j}$.
\end{theorem}
\noindent Theorem~\ref{thm:manifold flag_h} can be seen as a common generalization of manifold GLBI~\cite{NS} and balanced manifold GLBI~\cite{JMNS}.
A combinatorial corollary for doubly Cohen-Macaulay complexes is the following generalization of Theorem~\ref{thm:flag-h}.
\begin{theorem} \label{thm:2CM_flag h}
An $\bm{a}$-balanced doubly Cohen-Macaulay complexes $(\Delta,\kappa)$ over $\mathbb{F}_2$ satisfies the same inequality as in Theorem~\ref{thm:flag-h}.
\end{theorem}

Cook, Juhnke-Kubitzke, Murai and Nevo~\cite{CJMN} investigated whether an $\mathbb{N}^m$-graded Artinian reduction of the Stanley-Reisner ring of an $\bm{a}$-balanced simplicial sphere with $\bm{a} \in \mathbb{N}^m_+$ satisfies weak/strong Lefschetz property as an $\mathbb{N}$-graded algebra.
We show in Theorem~\ref{thm:fullrank at ends} that, in the generic $\mathbb{N}^m$-graded Artinian reduction of an $\bm{a}$-balanced simplicial sphere, the multiplication of a generic linear element is full-rank ``at ends''.
On the other hand, generalizing counterexamples given in~\cite{CJMN,Oba}, for any positive integer $i,d$ with $i<\frac{d}{2}$, we construct $(d-i,i)$-balanced simplicial spheres whose $\mathbb{N}^2$-graded Artinian reduction fails to have the weak Lefschetz property as an $\mathbb{N}$-graded algebra (Theorem~\ref{thm:counterexample}).

The paper is organized as follows.
After preliminaries are given in Section 2, Lee's formula for the evaluation map is recalled in Section 3.
In Section 4, generic $\mathbb{N}^m$-graded Artinian reduction is defined, and differential formula for the evaluation map in multigraded setting is derived.
Theorem~\ref{thm:main multiHL} for characteristic $2$ field is proved via anisotropy in Section 5, and the proof of Theorem~\ref{thm:main multiHL} for characteristic $0$ field is given in Section 6.
In Section 7, generalizations to manifold, simplicial cycles and doubly Cohen-Macaulay complexes are discussed.
In Section 8, Lefschetz property as an $\mathbb{N}$-graded algebra is discussed.

\section{Preliminaries}
We highlight some definitions and notations we use (see \cite{Sta-book} for general reference).
\subsection{Simplicial Complex}
Throughout, by a simplicial complex, we always mean an abstract simplicial complex, i.e., a downward closed collection of subsets of a finite set.
The vertex set of a simplicial complex $\Delta$ is denoted by $V(\Delta)$.
For a $(d-1)$-dimensional simplicial complex $\Delta$, the \emph{$f$-vector} of $\Delta$ is an integer vector $f(\Delta)=(f_{-1},\ldots,f_{d-1})$, where $f_i$ is the number of $i$-dimensional faces of $\Delta$, and the \emph{$h$-vector} of $\Delta$ is an integer vector $h(\Delta)=(h_0,\ldots,h_d)$ defined by 
\[
h_i = \sum_{j=0}^i (-1)^{i-j} \binom{d-j}{i-j} f_{j-1} \qquad \text{ for $i=0,\ldots,d$}.
\]

We denote the set of nonnegative (resp.~positive) integers by $\mathbb{N}$ (resp.~$\mathbb{N}_+$).
For $\bm{a}=(a_1,\ldots,a_m) \in \mathbb{N}^m$, we denote $|\bm{a}|=a_1+\cdots+a_m$.
Recall that for $\bm{a}=(a_1,\ldots,a_m) \in \mathbb{N}^m_+$ with $|\bm{a}|=d$, a pair $(\Delta,\kappa)$ of a $(d-1)$-dimensional simplicial complex $\Delta$ and a map $\kappa:V(\Delta)\rightarrow [m]=\{1,\ldots,m\}$ is $\bm{a}$-balanced if $|\tau \cap \kappa^{-1}(j)| \leq a_j$ holds for any $\tau \in \Delta$ and $j \in [m]$. Such a map $\kappa$ is called a \emph{coloring} of $\Delta$.
The pair $(\Delta,\kappa)$ is called an $\bm{a}$-balanced simplicial complex (though it is a pair), and we often say that $(\Delta,\kappa)$ satisfies a certain property if $\Delta$ satisfies it.
Note that a $(d-1)$-dimensional simplicial complex is $(1,\ldots,1)$-balanced if its $1$-skeleton is $d$-colorable (such a case is also called \emph{completely balanced} or just \emph{balanced} in the literature). Note also that any $(d-1)$-dimensional simplicial complex is $(d)$-balanced with the monochromatic coloring.
Let $(\Delta,\kappa)$ be an $\bm{a}$-balanced simplicial complex with $\bm{a} \in \mathbb{N}_+^m$.
The \emph{flag $f$-vector} of $(\Delta,\kappa)$ is an $m$-dimensional array $(f_{\bm{b}})_{\bm{0}\leq\bm{b}\leq\bm{a}}$ where $f_{\bm{b}}$ is the number of faces $\sigma \in \Delta$ with $|\sigma \cap \kappa^{-1}(j)|=b_j$ for $j=1,\ldots,m$. 
The \emph{flag $h$-vector} of $(\Delta,\kappa)$ is an $m$-dimensional array $(h_{\bm{b}})_{\bm{0}\leq\bm{b}\leq\bm{a}}$ defined by
\[
h_{\bm{b}}=\sum_{\bm{0} \leq \bm{c}\leq\bm{b}}(-1)^{|\bm{b}|-|\bm{c}|}\binom{\bm{a}-\bm{c}}{\bm{b}-\bm{c}}f_{\bm{c}} \qquad \text{ for all $\bm{b} \in \mathbb{N}^m$ with $\bm{b} \leq \bm{a}$},
\]
where $\binom{\bm{c}}{\bm{b}}=\prod_{j=1}^m \binom{c_j}{b_j}$.
These vectors refine the usual $f$- and $h$-vector in the sense that $f_{i-1}=\sum_{\bm{b} \leq \bm{a}, |\bm{b}|=i} f_{\bm{b}}$ and $h_{i}=\sum_{\bm{b} \leq \bm{a}, |\bm{b}|=i} h_{\bm{b}}$ hold for $i=0,\ldots,d$.

\subsection{Stanley-Reisner ring}
For an $\mathbb{N}^m$-graded module $M$ and $\bm{b} \in \mathbb{N}^m$, we denote by $M_{\bm{b}}$ the submodule of all homogeneous elements of degree $\bm{b}$.

Let $\mathbbm{k}$ be a field and let $\Delta$ be a simplicial complex. Let us denote by $\mathbbm{k}[\bm{x}]$ the polynomial ring $\mathbbm{k}[x_v:v \in V(\Delta)]$.
The Stanley-Reisner ring of $\Delta$ over $\mathbbm{k}$ is  $\mathbbm{k}[\Delta]=\mathbbm{k}[\bm{x}]/I_\Delta$, where $I_\Delta$ is the ideal generated by $x_\tau=\prod_{v \in \tau} x_v$ over all $\tau \not\in \Delta$.
It is known that the Stanley-Reisner ring of $\Delta$ has Krull dimension $\dim \Delta+1$.
For a $(d-1)$-dimensional simplicial complex $\Delta$, a length $d$ sequence of linear forms $\Theta=(\theta_1,\ldots,\theta_d)$ of $\mathbbm{k}[\Delta]$ is called a \emph{linear system of parameters} (\emph{l.s.o.p.}~for short) for $\mathbbm{k}[\Delta]$ if $\mathbbm{k}[\Delta]/(\Theta)=\mathbbm{k}[\Delta]/(\theta_1,\ldots,\theta_d)$ is a finite dimensional $\mathbbm{k}$-vector space.
The resulting quotient algebra $\mathbbm{k}[\Delta]/(\Theta)$ is called an \emph{Artinian reduction} of $\mathbbm{k}[\Delta]$ with respect to $\Theta$, and we usually denote it as $A(\Delta)$ or simply as $A$.
It is known that if $\mathbbm{k}$ is an infinite field, $\mathbbm{k}[\Delta]$ always has an l.s.o.p.

For an $\bm{a}$-balanced simplicial complex $(\Delta,\kappa)$ with $\bm{a} \in \mathbb{N}^m$, the polynomial ring $\mathbbm{k}[\bm{x}]$ naturally has the $\mathbb{N}^m$-grading, sometimes called the \emph{fine grading}, defined by $\deg x_v = \bm{e}_{\kappa(v)}$, where $\bm{e}_j \in \mathbb{N}^m$ denotes the $j$-th unit coordinate vector.
For an $\bm{a}$-balanced simplicial complex $(\Delta,\kappa)$, we say that a system of parameters $\Theta$ for $\mathbbm{k}[\Delta]$ is \emph{$\mathbb{N}^m$-graded} (or \emph{$\mathbb{N}^m$-homogeneous} or \emph{$\bm{a}$-colored}) if each $\theta_i$ is homogeneous in the fine $\mathbb{N}^m$-grading of $\mathbbm{k}[\Delta]$.
Stanley~\cite[Theorem 4.1]{Sta79} showed that if $\mathbbm{k}$ is an infinite field, every $\bm{a}$-balanced simplicial complex $(\Delta,\kappa)$ has an $\mathbb{N}^m$-graded l.s.o.p.~$\Theta$ for $\mathbbm{k}[\Delta]$, and $(\mathbbm{k}[\Delta]/(\Theta))_{\bm{b}}=0$ unless $\bm{0} \leq \bm{b} \leq \bm{a}$.
One can easily check that for an $\mathbb{N}^m$-graded l.s.o.p.~$\Theta$ for the Stanley-Reisner ring of an $\bm{a}$-balanced simplicial complex, $\Theta$ contains exactly $a_j$ elements of degree $\bm{e}_j$ for each $j$.

\subsection{Homological properties}
The \emph{link} of a face $\tau \in \Delta$ is $\lk_\tau(\Delta)=\{\sigma \in \Delta:\sigma \cap \tau = \emptyset, \sigma \cup \tau \in \Delta\}$. The (closed) \emph{star} of a face $\tau \in \Delta$ is $\st_\tau(\Delta)=\{\sigma \in \Delta: \sigma \cup \tau \in \Delta\}$.
For $W \subseteq V(\Delta)$, let $\Delta -W=\{\tau \in \Delta:\tau \not\supseteq W\}$.

A simplicial complex $\Delta$ is called \emph{Cohen-Mcaulay} over $\mathbbm{k}$ if there is an l.s.o.p.~$(\theta_1,\ldots,\theta_d)$ for $\mathbbm{k}[\Delta]$ such that $\mathbbm{k}[\Delta]$ is a free $\mathbbm{k}[\theta_1,\ldots,\theta_d]$-module.
By Reisner's theorem, a simplicial complex $\Delta$ is Cohen-Macaulay over $\mathbbm{k}$ if and only if it is pure and, for every face $\sigma \in \Delta$, $\widetilde{H}_i(\lk_\sigma(\Delta);\mathbbm{k})=0$ for all $i \neq \dim\Delta - |\sigma|$ (see~\cite[Corollary II.4.2]{Sta-book}). 
Here $\widetilde{H}_*(\Delta;\mathbbm{k})$ denotes the reduced simplicial homology group of $\Delta$ with coefficients in $\mathbbm{k}$.
Note that for an $\bm{a}$-balanced Cohen-Macaulay complex $(\Delta,\kappa)$, the equality $\dim (\mathbbm{k}[\Delta]/(\Theta))_{\bm{b}} = h_{\bm{b}}$ holds for $\bm{0} \leq \bm{b} \leq \bm{a}$.

For an $\mathbb{N}$-graded $\mathbbm{k}[\bm{x}]$-module $M$, its \emph{socle} is the submodule $\Soc(M)=\{a \in M: \mathfrak{m}a=0\}$, where $\mathfrak{m}=(x_1,\ldots,x_n)$ is the maximal graded ideal of $\mathbbm{k}[\bm{x}]$.
An $\mathbb{N}$-graded $\mathbbm{k}$-algebra of Krull dimension zero is said to be \emph{Gorenstein} if its socle is a one-dimensional $\mathbbm{k}$-vector space.
Note that an $\mathbb{N}$-graded finitely generated standard $\mathbbm{k}$-algebra $A=A_0\oplus \cdots \oplus A_d$ with $A_d \neq 0$ is Gorenstein if and only if $\dim A_d =1$ and the multiplication map $A_i \times A_{d-i} \rightarrow A_d \overset{\cong}{\rightarrow} \mathbbm{k}$ is a nondegenerate bilinear pairing for $i=0,\ldots,d$~\cite[Lemma 36]{AY}.

We say that a $(d-1)$-dimensional simplicial complex $\Delta$ is a \emph{simplicial $(d-1)$-sphere} if its geometric realization is homeomorphic to $\mathbb{S}^{d-1}$. A $(d-1)$-dimensional simplicial complex $\Delta$ is a \emph{homology $(d-1)$-sphere} over $\mathbbm{k}$ if $\widetilde{H}_*(\lk_{\tau}\Delta;\mathbbm{k}) \cong \widetilde{H}_*(\mathbb{S}^{d-|\tau|-1};\mathbbm{k})$ for every face $\tau \in \Delta$.
If $\Delta$ is a homology sphere over $\mathbbm{k}$, an Artinian reduction $A=\mathbbm{k}[\Delta]/(\Theta)$ is Gorenstein~\cite[Theorem II.5.1]{Sta-book}.
A simplicial complex $\Delta$ is a \emph{homology $(d-1)$-manifold} over $\mathbbm{k}$ if $\widetilde{H}_*(\lk_{\tau}\Delta;\mathbbm{k}) \cong \widetilde{H}_*(\mathbb{S}^{d-|\tau|-1};\mathbbm{k})$ for every \emph{nonempty} face $\tau \in \Delta$.

A pure $(d-1)$-dimensional simplicial complex is \emph{strongly connected} if for every pair of facets $\sigma$ and $\tau$ of $\Delta$, there is a sequence of facets $\sigma=\sigma_0,\sigma_1,\ldots, \sigma_m=\tau$ such that $|\sigma_{i-1} \cap \sigma_i|=d-1$ for $i=1,\ldots,m$.
A \emph{$(d-1)$-pseudomanifold} (without boundary) is a strongly connected pure $(d-1)$-dimensional simplicial complex such that every $(d-2)$-face is contained in exactly two facets.
A $(d-1)$-pseudomanifold is \emph{orientable} over $\mathbbm{k}$ if $\widetilde{H}_{d-1}(\Delta;\mathbbm{k})\cong \mathbbm{k}$.

\subsection{Lefschetz properties}
An Artinian Gorenstein standard $\mathbb{N}$-graded algebra $A=A_0\oplus\cdots\oplus A_d$ with $A_0 \simeq A_d \simeq \mathbbm{k}$ is said to have the \emph{weak Lefschetz property} if there is a linear element $\ell \in A_1$ such that the multiplication map $\times\ell:A_i \rightarrow A_{i+1}$ is either injective or surjective (or both) for $i=0,\ldots,d-1$.
It is said to have the \emph{strong Lefschetz property} if there is a linear element $\ell \in A_1$ such that the multiplication map $\times\ell^{d-2i}:A_i \rightarrow A_{d-i}$ is an isomorphism for all $i \leq \frac{d}{2}$.
For a general reference, see~\cite{HMMNWW}.

\section{Lee's formula for the evaluation map}
Let $\mathbbm{k}$ be a field of arbitrary characteristic, and let $\Delta$ be a $(d-1)$-dimensional simplicial complex.
Let $A=\mathbbm{k}[\Delta]/(\Theta)=A_0\oplus\cdots\oplus A_d$ be an Artinian reduction of $\mathbbm{k}[\Delta]$ with respect to an l.s.o.p.~$\Theta$.
Then, by \cite[Corollary 3.2]{TW}, $A_d$ is linearly isomorphic to $\widetilde{H}_{d-1}(\Delta;\mathbbm{k})$.
Thus, for a $(d-1)$-pseudomanifold $\Delta$ (without boundary) orientable over $\mathbbm{k}$, $A_d$ is a one-dimensional linear space. 
The linear isomorphism $\Psi:A_d \overset{\cong}{\rightarrow} \mathbbm{k}$ which is uniquely determined up to scaling is called the \emph{evaluation map} (or \emph{degree map, volume map, Brion's isomorphism}).
Lee~\cite{Lee} gave an explicit description of the evaluation map $\Psi$ with the appropriate scaling. (Although Lee's description is originally over $\mathbbm{k}=\mathbb{R}$, it readily extends to an arbitrary field. See also an equivalent description by Karu and Xiao~\cite{KX}.)
As the formula with the appropriate normalization plays an important role, we recall the formula below.

Let us prepare some conventions and notation used throughout the paper.
We always assume that $V(\Delta)=[n]:=\{1,\ldots,n\}$ and let $\mathbbm{k}[\bm{x}]=\mathbbm{k}[x_1,\ldots,x_n]$. For a sequence $J=(v_1,\ldots,v_k)$ of vertices (possibly with repetitions), we denote $x_J=x_{v_1}\cdots x_{v_k}$.
We abbreviate the projection from $\mathbbm{k}[\bm{x}]$ to an Artinian reduction $A$ of $\mathbbm{k}[\Delta]$ as long as it is not confusing.
So, for example, the composite $\mathbbm{k}[\bm{x}]_d \twoheadrightarrow A_d \overset{\cong}{\rightarrow} \mathbbm{k}$ is also denoted as $\Psi$.
An l.s.o.p.~$\Theta=(\theta_1,\ldots,\theta_d)$ for $\mathbbm{k}[\Delta]$ is identified with a map $p:V(\Delta) \rightarrow \mathbbm{k}^d$ through the relation $\theta_k=\sum_{v\in V(\Delta)} p(v)_k x_v$ for $k=1,\ldots,d$. The map $p$ is called a \emph{point configuration}.
For a positively oriented facet $\sigma=[v_1,\ldots,v_d]$ of $\Delta$, let $[\sigma]=\det \begin{pmatrix} p(v_1) & \cdots & p(v_d) \end{pmatrix}$.

We also need the following notation to state Lee's formula.
Let $v^*$ be a new vertex not in $V(\Delta)$ with an associated position $p'(v^*) \in \mathbbm{k}^d$, and for a positively oriented facet $\sigma=[v_1,\ldots,v_d]$, let $[\sigma-v_i+v^*]$ be the determinant of the matrix obtained by replacing the $i$-th column of the matrix $\begin{pmatrix} p(v_1) & \cdots & p(v_d) \end{pmatrix}$ with $p'(v^*)$.
Here, $p'(v^*)$ has to be in sufficiently general position so that none of $[\sigma-v_i+v^*]$ vanishes. (One may need to extend the field to choose such a vector when $\mathbbm{k}$ is a finite field.)

Now we are ready to state Lee's formula.
\begin{lemma}{\cite[Lemma 3.1, Theorem 3.2]{KX}} \label{lem:Lee formula}
Let $\Delta$ be an orientable $(d-1)$-pseudomanifold over a field $\mathbbm{k}$.
Fix an orientation of the facets of $\Delta$.
Let $A$ be an Artinian reduction of $\mathbbm{k}[\Delta]$ with respect to an l.s.o.p.~$\Theta$, and let $\Psi:A_d \rightarrow \mathbbm{k}$ be the evaluation map.
Then, under a suitable normalization, the following hold:
\begin{itemize}
    \item[(i)] For any positively oriented facet $\sigma \in \Delta$, we have $\Psi(x_\sigma)=\frac{1}{[\sigma]}$.
    \item[(ii)] More generally, for any length $d$ sequence of vertices $J=(v_1,\ldots,v_d)$, we have
\begin{align} 
\Psi(x_J)&=\sum_{\sigma \in \Delta:\text{ facet containing }\supp(x_J)} \frac{1}{[\sigma]} \frac{\prod_{k=1}^{d}[\sigma + v^* - v_k]}{\prod_{v \in \sigma} [\sigma+v^*-v]} .\label{eq:eval}
\end{align}
\end{itemize}
Here, the sum is taken over all positively oriented facets of $\Delta$ containing $\supp(x_J):=\{v_1,\ldots,v_n\}$.
\end{lemma}
Throughout the paper, we assume that the evaluation map $\Psi$ is normalized so that Lemma~\ref{lem:Lee formula}(i) holds.
We note that although we need the position $p'(v^*)$ of a new vertex to explicitly write the formula~(\ref{eq:eval}), the value after the computation of the right hand side of (\ref{eq:eval}) is independent of the choice of $p'(v^*)$.

\begin{remark}
As a side note, we describe how the formula (\ref{eq:eval}) is derived from the work of Lee~\cite{Lee} for the case where $\mathbbm{k}=\mathbb{R}$.
Define an inner product between degree $d$ homogeneous polynomials $a(\bm{x})=\sum_{|\bm{r}|=d} a_{\bm{r}} \frac{\bm{x}^{\bm{r}}}{\bm{r}!}$ and $b(\bm{x})=\sum_{|\bm{r}|=d} b_{\bm{r}} \frac{\bm{x}^{\bm{r}}}{\bm{r}!}$ by
\begin{equation*}
\langle a(\bm{x}), b(\bm{x}) \rangle = \sum_{|\bm{r}|=d} a_{\bm{r}}b_{\bm{r}}.
\end{equation*}
The orthogonal complement of $(I_\Delta + (\Theta))_d :=(I_\Delta + (\Theta)) \cap \mathbbm{k}[\bm{x}]_d$ is called the linear $d$-stress space of $(\Delta,p)$, where $p$ is the point configuration associated to $\Theta$.
The space of linear $d$-stress is linearly isomorphic to $A_d$, and thus there is a base $\gamma$, unique up to scaling, of the linear $d$-stress.
The map $\mathbbm{k}[\bm{x}]_d \rightarrow \mathbbm{k}; a \mapsto \langle a,\gamma \rangle$ is a nonzero linear function that vanishes on $(I_\Delta + (\Theta))_d$. 
Thus, $\Psi$ (considered as a function over $\mathbbm{k}[\bm{x}]_d$) coincides with this map (up to scaling). 
It remains to describe the coefficient of the canonical linear $d$-stress $\gamma$.
The coefficients of the squarefree terms of $\gamma$ are given in \cite[Proof of Theorem 14]{Lee}, and the coefficients of the non-squarefree terms are then given by \cite[Theorem 11]{Lee}. These agree with the formula (\ref{eq:eval}).
\end{remark}

\section{Generic $\mathbb{N}^m$-graded Artinian reduction and differential formula}
\subsection{generic $\mathbb{N}^m$-graded Artinian reduction}
For an $\bm{a}$-balanced simplicial complex $(\Delta,\kappa)$ with $\bm{a} \in \mathbb{N}^m_+$ and $|\bm{a}|=d$, we define the generic $\mathbb{N}^m$-graded Artinian reduction of $\mathbbm{k}[\Delta]$ as follows.
Fix a partition $\mathcal{I}_1 \sqcup \cdots \sqcup \mathcal{I}_m$ of $[d]$ with $|\mathcal{I}_j|=a_j$ for $j=1,\ldots,m$.
Consider the set of new auxiliary indeterminates
\begin{equation*}
\{p_{k,v}: k \in [d], v \in V(\Delta), k \in \mathcal{I}_{\kappa(v)}\}
\end{equation*}
and let $\widetilde{\mathbbm{k}}=\mathbbm{k}(p_{k,v})$ be the rational function field of these indeterminates with coefficients in $\mathbbm{k}$. 
Define the $\mathbb{N}^m$-graded l.s.o.p.~$\Theta=(\theta_1,\ldots,\theta_d)$ by 
\[
\begin{pmatrix}
\theta_1 \\
\vdots \\
\theta_d
\end{pmatrix}
=
\bm{P}
\begin{pmatrix}
x_1 \\
\vdots \\
x_n
\end{pmatrix},
\]
where the $(k,v)$-th entry of the coefficient matrix $\bm{P}$ is $p_{k,v}$ if $k \in \mathcal{I}_{\kappa(v)}$ and $0$ otherwise.
The quotient $\mathbb{N}^m$-graded algebra $A=\widetilde{\mathbbm{k}}[\Delta]/(\Theta)$ is called the \emph{generic $\mathbb{N}^m$-graded Artinian reduction} of $\mathbbm{k}[\Delta]$ (with respect to a coloring $\kappa$ and $\bm{a}$).
Note that, when $m=1$, the generic $\mathbb{N}$-graded Artinian reduction coincides with the generic Artinian reduction in the sense of~\cite{PP}. We remark that to be consistent with the definition of~\cite{PP}, $A=\widetilde{\mathbbm{k}}[\Delta]/(\Theta)$ is called the generic $\mathbb{N}^m$-graded Artinian reduction of $\mathbbm{k}[\Delta]$, not of $\widetilde{\mathbbm{k}}[\Delta]$, though $A$ is the Artinian reduction of $\widetilde{\mathbbm{k}}[\Delta]$ in the usual sense.
By~\cite[Theorem 4.1]{Sta79}, as an $\mathbb{N}^m$-graded algebra, $A$ is decomposed into $\mathbb{N}^m$-homogeneous components as $A=\bigoplus_{\bm{0} \leq \bm{b} \leq \bm{a}} A_{\bm{b}}$.
The homogeneous decomposition as a coarse $\mathbb{N}$-graded algebra is denoted as $A=\bigoplus_{i=0}^d A_i$.

\subsection{Differential formula in characteristic $2$}
In the generic $\mathbb{N}$-graded Artinian reduction, the right-hand-side of (\ref{eq:eval}) in Lee's formula is a rational function of new indeterminates $p_{k,v}$. (As mentioned, it is independent of the position of the new vertex.)
Papadakis and Petrotou~\cite{PP} considered taking partial derivative of (\ref{eq:eval}) with respect to new indeterminates $p_{k,v}$, and they prove a remarkable formula in characteristic $2$.
This formula is later generalized by Karu and Xiao~\cite[Theorem 4.1]{KX}.
We recall this formula (see also~\cite{APP} for different formula that holds in arbitrary characteristic).

In this subsection, we assume that the field $\mathbbm{k}$ is of characteristic $2$. Then automatically every pseudomanifold is orientable.
For a $(d-1)$-pseudomanifold $\Delta$, let $A=\widetilde{\mathbbm{k}}[\Delta]/(\Theta)$ be the generic ($\mathbb{N}$-graded) Artinian reduction of $\mathbbm{k}[\Delta]$, where $\widetilde{\mathbbm{k}}=\mathbbm{k}(p_{kv}:k \in [d], v \in V(\Delta))$.
For a length $d$ sequence $I=(v_1,\ldots,v_d)$ of vertices, define the differential operator $\partial_I$ by $\partial_{p_{1,v_1}}\circ \cdots \circ \partial_{p_{d,v_d}}$, where $\partial_{p_{k,v}}$ denoted the (formal) partial derivative with respect to $p_{k,v}$. Under these notations, the following holds.
\begin{theorem}{\cite[Theorem 4.1]{KX}} \label{thm:KX}
Let $\Delta$ be a $(d-1)$-pseudomanifold, and let $\mathbbm{k}$ be a field of characteristic $2$.
Let $A=\widetilde{\mathbbm{k}}[\Delta]/(\Theta)$ be the generic $\mathbb{N}$-graded Artinian reduction of $\mathbbm{k}[\Delta]$, where $\widetilde{\mathbbm{k}}=\mathbbm{k}(p_{kv}:\, 1\leq k\leq d, \, v\in V(\Delta))$.
Let $\Psi:A_{d} \rightarrow \widetilde{\mathbbm{k}}$ be the evaluation map normalized as in~Lemma~\ref{lem:Lee formula}.
Then, for any length $d$ sequences $I$ and $J$ of vertices,
\[
\partial_I \Psi(x_J) =\Psi(\sqrt{x_Ix_J})^2
\]
holds.
Here, for a monomial $x_L$, define its square root $\sqrt{x_L}$ by $x_K$ if there is a monomial $x_K$ with $x_K^2=x_L$ and $0$ otherwise.
\end{theorem}

We generalize the formula in Theorem~\ref{thm:KX} in the setting of generic $\mathbb{N}^m$-graded Artinian reduction by a simple trick of substitution.
Let $(\Delta,\kappa)$ be an $\bm{a}$-balanced pseudomanifold and let $A=\widetilde{\mathbbm{k}}[\Delta]/(\Theta)$ be the generic $\mathbb{N}^m$-graded Artinian reduction of $\mathbbm{k}[\Delta]$, where $\widetilde{\mathbbm{k}}=\mathbbm{k}(p_{kv})$.
We call a length $d$ sequence of vertices $I=(v_1,\ldots,v_d)$ (possibly with repetition) \emph{$\kappa$-transversal} if $k \in \mathcal{I}_{\kappa(v_k)}$ for $k=1,\ldots,d$.
Note that $I=(v_1,\ldots,v_d)$ is a $\kappa$-transversal sequence if and only if there exist corresponding variables $p_{1,v_1}, \ldots, p_{d,v_d}$.
Note also that for every degree $\bm{a}$ monomial $x_J$ in $\mathbbm{k}[\bm{x}]$, $J$ can be reordered into a $\kappa$-transversal sequence.
For a $\kappa$-transversal sequence $I=(v_1,\ldots,v_d)$, define the differential operator $\partial_I$ by $\partial_{p_{1,v_1}}\circ \cdots \circ \partial_{p_{d,v_d}}$.
The following differential formula for the map $\Psi$ holds.
\begin{lemma} \label{lem:diff}
Let $(\Delta,\kappa)$ be an $\bm{a}$-balanced $(d-1)$-pseudomanifold for $\bm{a}\in\mathbb{N}^m_+$ with $|\bm{a}|=d$ and let $\mathbbm{k}$ be a field of characteristic $2$.
Let $A=\widetilde{\mathbbm{k}}[\Delta]/(\Theta)$ be the generic $\mathbb{N}^m$-graded Artinian reduction of $\mathbbm{k}[\Delta]$ with respect to $\kappa$ and $\bm{a}$.
Let $\Psi:A_{\bm{a}} \rightarrow \widetilde{\mathbbm{k}}$ be the evaluation map normalized as in~Lemma~\ref{lem:Lee formula}.
Then, for any $\kappa$-transversal sequence $I$ and $J$,
\[
\partial_I \Psi(x_J) =\Psi(\sqrt{x_Ix_J})^2
\]
holds.
Here, for a monomial $x_L$, define its square root $\sqrt{x_L}$ by $x_K$ if there is a monomial $x_K$ with $x_K^2=x_L$ and $0$ otherwise.
\end{lemma}
\begin{proof}
Denote the generic $\mathbb{N}$-graded Artinian reduction of $\mathbbm{k}[\Delta]$ by $A'=K[\Delta]/(\Theta)$, where $K=\mathbbm{k}(p_{kv}:k \in [d], v\in V(\Delta))$, and the corresponding normalized evaluation map by $\Psi':A'_d \rightarrow K$.
Let $R$ be the localization of $\mathbbm{k}[p_{kv}:k \in [d], v\in V(\Delta)]$ at the irreducible polynomials $\{p_{kv}:k \not\in \mathcal{I}_{\kappa(v)}\}$.
Fix a $\kappa$-transversal sequence $I$ and a length $d$ sequence $J$ of vertices.
By Theorem~\ref{thm:KX}, we have the identity
\begin{equation}\label{eq:KX}
    \partial_I(\Psi'(x_J)) = \Psi'(\sqrt{x_Ix_J})^2.
\end{equation}
Here, we have $\Psi'(x_J), \Psi(\sqrt{x_Ix_J}) \in R$.
This is because, in the right hand side of Lee's formula (\ref{eq:eval}), the denominators do not vanish after the substitutions $p_{kv}=0$ for all pairs of $(k,v)$ with $k \not\in \mathcal{I}_{\kappa(v)}$ by the Kind-Kleinschmidt's criterion on l.s.o.p.~for Stanley-Reisner ring~\cite[Lemma III.2.4]{Sta-book}.

The map $\xi:R \rightarrow \widetilde{\mathbbm{k}}$ defined by the substitution $p_{kv}=0$ for all $(k,v)$ with $k \not\in \mathcal{I}_{\kappa(v)}$ is a ring homeomorphism, and $\xi$ commutes with the partial derivative $\partial_{p_{k,v}}$ for any $(k,v)$ with $k\in \mathcal{I}_{\kappa(v)}$.
As we have $\xi \circ \Psi'(x_L)=\Psi(x_L)$ for any total degree $d$ monomial $x_L$, we can deduce the desired identity from~(\ref{eq:KX}).
\end{proof}
Lemma~\ref{lem:diff} can be readily strengthened as below.
\begin{corollary} \label{cor:diff general}
Let $(\Delta,\kappa)$, $d$, $A$, $\Psi$ be as in Lemma~\ref{lem:diff}. 
For a $\kappa$-transversal sequence $I$, an element $g \in A_i$ with $i \leq \frac{d}{2}$, and a length $d-2i$ sequence $J$ of vertices,
\[
\partial_I \Psi(g^2x_J) =\Psi(g\sqrt{x_Ix_J})^2
\]
holds.
\end{corollary}
\begin{proof}
Writing $g=\sum_{K} \lambda_K x_K$ ($\lambda_K \in \widetilde{\mathbbm{k}}$)\footnote{Recall that we are abbreviating the projection from the polynomial ring to $A$.}, we have
\begin{align*}
\partial_I \Psi(g^2x_J) &= \partial_I \Psi\left(\sum_{K} \lambda_K^2 x_K^2 x_J\right) & (\text{by }\left(\sum_{K} \lambda_K x_K\right)^2 = \sum_{K} \lambda_K^2 x_K^2 \text{ in characteristic $2$}) \\
&=\sum_{K} \partial_I (\lambda_K^2 \Psi(x_K^2x_J)) & (\text{by linearity of $\Psi, \partial_I$}) \\
&=\sum_{K} \lambda_K^2 \partial_I  \Psi(x_K^2x_J) & (\text{by } \partial_{p_{k,v}}(f^2g)=f^2\partial_{p_{k,v}}(g) \text{ for $f,g \in \widetilde{\mathbbm{k}}$ in characteristic $2$}) \\
&=\sum_{K} \lambda_K^2 \Psi(x_K \sqrt{x_Ix_J})^2 & (\text{by Lemma~\ref{lem:diff}}) \\
&=\Psi(g \sqrt{x_Ix_J})^2.
\end{align*}
\end{proof}

\section{Proof of of Theorem~\ref{thm:main multiHL} via anisotropy}
Throughout this section, we assume that $\mathbbm{k}$ is a field of characteristic $2$ and $(\Delta,\kappa)$ is an $\bm{a}$-balanced homology sphere over $\mathbb{F}_2$ for $\bm{a} \in \mathbb{N}^m_+$ with $|\bm{a}|=d$. 
Let $A=\widetilde{\mathbbm{k}}[\Delta]/(\Theta)$ be the generic $\mathbb{N}^m$-graded Artinian reduction of $\mathbbm{k}[\Delta]$, where $\widetilde{\mathbbm{k}}=\mathbbm{k}(p_{kv})$.
By Gorensteiness, the multiplication map $A_i \times A_{d-i} \rightarrow A_d \overset{\Psi}{\rightarrow} \widetilde{\mathbbm{k}}$ is a nondegenerate for each $0 \leq i\leq d$.
Hence, the multiplication map $A_{\bm{b}} \times A_{\bm{a}-\bm{b}} \rightarrow A_{\bm{a}} \overset{\Psi}{\rightarrow} \widetilde{\mathbbm{k}}$ is nondegenerate for each $\bm{b}\in\mathbb{N}^m$ with $\bm{b} \leq \bm{a}$.
We call this property as \emph{multigraded Poincar\'{e} duality}.

Our proof of Theorem~\ref{thm:main multiHL} relies on anisotropy technique used in~\cite{APP,APPS,KX,PP}.
For a vector space $W$ over a field $\mathbbm{k}$, a bilinear form $\varphi:W \times W \rightarrow \mathbbm{k}$ is \emph{anisotropic} if $\varphi(u,u) \neq 0$ holds for any nonzero $u \in W$.
Note that a bilinear form is anisotropic if and only if the restriction $\varphi|_{W' \times W'}$ is nondegenerate for any nonzero subspace $W'$ of $W$.
We prove the following combination of anisotropy and multigraded strong Lefschetz property in a field of characteristic $2$ with the explicit Lefschetz elements.
\begin{theorem} \label{thm:main anis}
Let $(\Delta,\kappa)$ be an $\bm{a}$-balanced homology sphere over $\mathbb{F}_2$, and let $\mathbbm{k}$ be a field of characteristic $2$. 
Let $A=\widetilde{\mathbbm{k}}[\Delta]/(\Theta)$ be the generic $\mathbb{N}^m$-graded Artinian reduction of $\mathbbm{k}[\Delta]$.
Define $\ell_j=\sum_{v \in \kappa^{-1}(j)} x_v \in A_{\bm{e}_j}$ for $j=1,\ldots,m$.
Then, for any $\bm{b} \in \mathbb{N}^m$ with $\bm{b} \leq \frac{\bm{a}}{2}$, the bilinear form $\mathcal{Q}:A_{\bm{b}}\times A_{\bm{b}} \rightarrow \widetilde{\mathbbm{k}}$ defined by
\[
\mathcal{Q}(g,h)=\Psi(gh \bm{\ell}^{\bm{a}-2\bm{b}})
\]
is anisotropic, where $\Psi:A_{\bm{a}} \rightarrow \widetilde{\mathbbm{k}}$ is the evaluation map.
\end{theorem}

Toward the proof of Theorem~\ref{thm:main anis}, we first prove an auxiliary lemma, which can be seen as the combination of a multigraded version of weak Lefschetz property and anisotropy.
\begin{lemma} \label{lemma:WLP anis}
Let $(\Delta,\kappa)$, $\bm{a}$, $\mathbbm{k}$, $A$, $\ell_j$ be as in Theorem~\ref{thm:main anis}.
Let $S$ be a (possibly empty) subset of $[m]$ and let $\bm{e}_S=\sum_{j \in S} \bm{e}_j \in \mathbb{N}^m$ be the characteristic vector of $S$.
For $\bm{b} \in \mathbb{N}^m$ with $2\bm{b}+\bm{e}_S \leq \bm{a}$, define the bilinear form $\mathcal{Q}' : A_{\bm{b}}\times A_{\bm{b}} \rightarrow A_{2\bm{b}+\bm{e}_S}$ by
\[
\mathcal{Q}'(g,h)=gh \bm{\ell}^{\bm{e}_S}.
\]
Then $\mathcal{Q}'(g,g)\neq 0$ for any nonzero $g \in A_{\bm{b}}$,.
\end{lemma}
\begin{proof}
Suppose that $g$ is a nonzero element of $A_{\bm{b}}$.
As $A_{\bm{a}-\bm{b}}$ is generated by monomials, by multigraded Poincar\'{e} duality of $A$, there is a monomial $x_K$ of degree $\bm{a}-\bm{b}$ such that $gx_K \neq 0$ in $A_{\bm{a}}$. Its square $x_K^2$ is of degree $2\bm{a}-2\bm{b}$, where $2\bm{a}-2\bm{b} \geq \bm{a}+\bm{e}_S$ by assumption.
Hence there is a $\kappa$-transversal sequence $I$ and a set of vertices $U^* \in V_S:=\prod_{j \in S} \kappa^{-1}(j)$ and a length $d-2|\bm{b}|-|S|$ sequence of vertices $J$ satisfying $x_K^2=x_Ix_{U^*}x_J$\footnote{Since, for any degree $\bm{a}$ monomial $x_L$, $L$ can be reordered into $\kappa$-transversal sequence, the desired decomposition $x_K^2=x_Ix_Ux_J$ is obtained by assigning variables in greedy way.}.

Now we have the following identity:
\begin{align}
\partial_I\Psi(\mathcal{Q}'(g,g)x_J) &= \sum_{U \in V_S} \partial_I \Psi(g^2x_Ux_J) & \text{(by linearity of $\Psi$, $\partial_I$ and $\bm{\ell}^{\bm{e}_S}=\sum_{U \in V_S} x_U$)} \notag \\
&=\sum_{U \in V_S} \Psi(g\sqrt{x_Ix_Ux_J})^2 &\text{(by Corollary~\ref{cor:diff general})} \notag \\
&\overset{(*)}{=}\Psi(g\sqrt{x_Ix_{U^*}x_J})^2 = \Psi(g x_K)^2 \label{eq:WLP}
\end{align}
Here, in $(*)$, we are using the fact that, by the definition of square root, for a fixed monomial $x_Ix_J$, there is a unique squarefree monomial $x_{U'}$ with $\sqrt{x_Ix_{U'}x_J} \neq 0$.
By our choice of $U^*$, this is achieved by $x_{U'}=x_{U^*}$. 
As monomials $x_U$ for $U \in V_S$ are all distinct and squarefree, the equality $(*)$ holds.
Now, $g x_K$ is a nonzero element in $A_{\bm{a}}$ and $\Psi$ is an isomorphism, so we have $\Psi(g x_K)^2 \neq 0$.
Hence, by the identity (\ref{eq:WLP}), $\partial_I\Psi(\mathcal{Q}'(g,g) x_J)$ must be nonzero. Therefore $\mathcal{Q}'(g,g)$ is nonzero.
\end{proof}

Now we are ready to prove Theorem~\ref{thm:main anis}.
\begin{proof}[Proof of Theorem~\ref{thm:main anis}]
Suppose that $\mathcal{Q}(g,g)=0$ holds for $g \in A_{\bm{b}}$. As $\Psi$ is an isomorphism, we have $g^2\bm{\ell}^{\bm{a}-2\bm{b}}=0$.
By applying Lemma~\ref{lemma:WLP anis} for 
\[
g\prod_{j \in [m]} \ell_j^{\left\lfloor \frac{a_j-2b_j}{2}\right\rfloor}
\]
and $S=\{j \in [m]:a_j-2b_j \text{ is odd}\}$, we have
\begin{equation}
g\prod_{j \in [m]} \ell_j^{\left\lfloor \frac{a_j-2b_j}{2}\right\rfloor}=0. \label{eq:main proof}
\end{equation}
By multiplying $g$ to both sides of (\ref{eq:main proof}), we obtain  
\[
g^2\prod_{j \in [m]} \ell_j^{\left\lfloor \frac{a_j-2b_j}{2}\right\rfloor}=0.
\]
By repeating in this way, we can reduce the power of $\ell_j$s' and we eventually obtain $g=0$.
\end{proof}
Now Theorem~\ref{thm:main multiHL} for characteristic $2$ is immediate.
\begin{proof}[Proof of Theorem~\ref{thm:main multiHL} for characteristic $2$]
Suppose that the field $\mathbbm{k}$ is of characteristic $2$.
Define the Lefschetz elements $\ell_j$ for $j=1,\ldots,m$ as in Theorem~\ref{thm:main anis}.
Then, Theorem~\ref{thm:main anis} implies that the linear map $\times \bm{\ell}^{\bm{a}-2\bm{b}}:A_{\bm{b}} \rightarrow A_{\bm{a}-\bm{b}}$ is injective for every $\bm{b} \leq \frac{\bm{a}}{2}$.
By multigraded Poincar\'{e} duality of $A$, we have $\dim A_{\bm{b}}=\dim A_{\bm{a}-\bm{b}}$, and thus the map is an isomorphism.
\end{proof}
Theorem~\ref{thm:flag-h} is readily obtained as a corollary of Theorem~\ref{thm:main multiHL}.
\begin{proof}[Proof of Theorem~\ref{thm:flag-h}]
By Theorem~\ref{thm:main multiHL} over a field $\mathbbm{k}$ of characteristic $2$, the composite
\[
A_{\bm{b}} \overset{\times \bm{\ell}^{\bm{c}-\bm{b}}}{\longrightarrow} A_{\bm{c}} \overset{\times \bm{\ell}^{\bm{a}-\bm{b}-\bm{c}}}{\longrightarrow} A_{\bm{a}-\bm{b}}
\]
is an isomorphism. So, the linear map $\times \bm{\ell}^{\bm{c}-\bm{b}}:A_{\bm{b}}\rightarrow A_{\bm{c}}$ is injective.
Thus, $h_{\bm{b}} = \dim A_{\bm{b}} \leq \dim A_{\bm{c}} = h_{\bm{c}}$ holds.
\end{proof}
Taking the weighted sum of Theorem~\ref{thm:flag-h}, we can prove that the $h$-vector of an $\bm{a}$-balanced homology sphere is multiplicatively increasing ``at appropriate ends''. 
More precisely, we have the following.
\begin{corollary} \label{cor:h-vec}
For an $\bm{a}$-balanced homology sphere $(\Delta,\kappa)$ over $\mathbb{F}_2$ with $\bm{a}=(a_1,\ldots,a_m) \in \mathbb{N}^m_+$, 
\[
\frac{h_i}{\binom{m+i-1}{i}} \leq \frac{h_{i+1}}{\binom{m+i}{i+1}}
\]
holds for every nonnegative integer $i \in \mathbb{N}$ with $i \leq \min_{j=1}^m \frac{a_j-1}{2}$.
\end{corollary}
\begin{proof}[Proof]
For any $\bm{b} \in \mathbb{N}^m$ with $|\bm{b}|=i$ and any $j \in[m]$, $2\bm{b}+\bm{e}_j \leq \bm{a}$ holds by the assumption $i \leq \min_{j=1}^m \frac{a_j-1}{2}$.
So, by Theorem~\ref{thm:flag-h}, we have $h_{\bm{b}} \leq h_{\bm{b}+\bm{e}_j}$.
Now the desired inequality follows from
\[
(m+i)\sum_{|\bm{b}|=i}h_{\bm{b}} = \sum_{|\bm{b}|=i}\sum_{j=1}^m (b_j+1) h_{\bm{b}} \leq \sum_{|\bm{b}|=i}\sum_{j=1}^m (b_j+1) h_{\bm{b}+\bm{e}_j} =(i+1)\sum_{|\bm{c}|=i+1}h_{\bm{c}}
\]
and the equality $h_{i'}=\sum_{|\bm{b}|=i'} h_{\bm{b}}$ for $i'=i,i+1$.
\end{proof}
We remark that, although balanced GLBI can be obtained from Theorem~\ref{thm:flag-h}, Corollary~\ref{cor:h-vec} is not a generalization of balanced GLBI (for when $\bm{a}=\bm{1}$ Corollary~\ref{cor:h-vec} only yields a trivial inequality $h_0 \leq h_1/m$).
There are other examples in which we can obtain the inequality about the $h$-vector from Theorem~\ref{thm:flag-h} by the same technique of grouping some colors and taking weighted sum as in~\cite{JM,KN}. 
For example, we can prove $\frac{k}{2} h_1 \leq h_2$ for a $2\bm{1}_k$-balanced simplicial sphere.
There still remain many open cases about the behavior of the $h$-vector of an $\bm{a}$-balanced simplicial sphere ``around the middle''.

\section{From characteristic $2$ to characteristic $0$}
In this section, we prove the multigraded strong Lefschetz property (Theorem~\ref{thm:main multiHL}) over a field of characteristic $0$ based on the result over a field of characteristic $2$.
Though the argument in this section may be well-known, we include it for completeness.

We start with a lemma about the basis (see also~\cite[Lemma 5.1]{KX}).
\begin{lemma} \label{lem:basis}
Let $(\Delta,\kappa)$ be an $\bm{a}$-balanced homology sphere over $\mathbb{F}_2$ and let $\mathbbm{k}$ be a field of characteristic $0$.
Let $\mathcal{B}$ be a set of monomials that forms a basis of the generic $\mathbb{N}^m$-graded Artinian reduction $\widetilde{\mathbb{F}_2}[\Delta]/(\Theta)$ of $\mathbb{F}_2[\Delta]$.
Then, $\mathcal{B}$ also forms a basis of the generic $\mathbb{N}^m$-graded Artinian reduction $A=\widetilde{\mathbbm{k}}[\Delta]/(\Theta)$ of $\mathbbm{k}[\Delta]$.
\end{lemma}
\begin{proof}
Let $d=|\bm{a}|$ and let $\widetilde{\mathbbm{k}}=\mathbbm{k}(p_{k,v}:(k,v) \in \mathcal{I})$, where $\mathcal{I}$ denotes the set of all new indeterminates used in the generic $\mathbb{N}^m$-generic Artinian reduction.
By~\cite[Lemma 2.1 (2)]{KX} and the fact that whether $\Delta$ is a homology sphere over a given field only depends on its characteristic, $\Delta$ is also a homology sphere over $\widetilde{\mathbbm{k}}$.
Hence, by Reisner's theorem, $\widetilde{\mathbb{F}_2}[\Delta]/(\Theta)$ and $A$ has the same dimension $h_0+\cdots+h_d$ as a $\widetilde{\mathbb{F}_2}$-vector space and a $\widetilde{\mathbbm{k}}$-vector space, respectively.
Thus, it suffices to prove that $\mathcal{B}$ is linearly independent in $A$.

Suppose on the contrary, $\mathcal{B}$ is linearly dependent in $A$.
Then, there is some number $D \in \mathbb{N}$ such that the finite set of elements
\begin{equation*}
S=\left\{b \bm{\theta}^{\bm{\alpha}}:b \in \mathcal{B}, \bm{\alpha} \in \mathbb{N}^d, \, \deg b + |\bm{\alpha}|=D \right\}
\end{equation*}
is linear dependent in $\widetilde{\mathbbm{k}}[\Delta]_D$. Here $\deg b$ denotes the degree of $b$ under the natural $\mathbb{N}$-grading and $\bm{\theta}^{\bm{\alpha}} = \prod_{j=1}^d \theta_j^{\alpha_j}$.
Let $M$ be the standard basis of $\widetilde{\mathbbm{k}}[\Delta]_D$ consisting of all monomials of degree $D$ whose support is contained in $\Delta$. 
Then, for each $s \in S$, there is unique $(t_{sm})_{m \in M}$ with $t_{sm} \in \mathbb{Z}[p_{k,v}:(k,v) \in \mathcal{I}]$ such that $s - \sum_{m \in M} t_{sm} m \in I_{\Delta}$.
Consider the $|S| \times |M|$ matrix $T=(t_{sm})$.
The linear dependence of $S$ implies that $T$ is row dependent over a field $\widetilde{\mathbbm{k}}$ of characteristic $0$.

On the other hand, since $\widetilde{\mathbb{F}_2}[\Delta]$ is a free $\widetilde{\mathbb{F}_2}[\theta_1,\ldots,\theta_d]$-module, $S$ is linearly independent in $\widetilde{\mathbb{F}_2}[\Delta]_D$.
This implies that the matrix $T \bmod 2$ is row independent over $\widetilde{F_2}$.
Thus, there is a row-full square submatrix of $T$ whose determinant (as a polynomial in $\mathbb{Z}[p_{k,v}:(k,v) \in \mathcal{I}]$) is nonzero modulo $2$. 
This contradicts to the fact that $T$ is row dependent over a field $\widetilde{\mathbbm{k}}$ of characteristic $0$.
\end{proof}
We now prove Theorem~\ref{thm:main multiHL} for a field of characteristic $2$.
\begin{proof}[Proof of Theorem~\ref{thm:main multiHL} for a characteristic $0$ field]
Let $\mathbbm{k}$ be a field of characteristic $0$ and let $A=\widetilde{\mathbbm{k}}[\Delta]/(\Theta)$.
Define Lefschetz elements as in Theorem~\ref{thm:main anis}.
As discussed in the proof Lemma~\ref{lem:basis}, $\Delta$ is also a homology sphere over $\widetilde{\mathbbm{k}}$.
Hence $\mathbbm{k}[\Delta]$ is a free $\mathbbm{k}[\theta_1,\ldots,\theta_d]$-module and $\dim A_{\bm{b}}=\dim A_{\bm{a}-\bm{b}}$ holds.
So it suffices to prove the injectivity of the linear map $\times\bm{\ell}^{\bm{a}-2\bm{b}}:A_{\bm{b}} \rightarrow A_{\bm{a}-\bm{b}}$.

For this, pick a basis $\mathcal{B}$ of $\widetilde{\mathbb{F}_2}[\Delta]/(\Theta)$ consisting of monomials (such a basis can be obtained by ordering all the monomials of $\mathbb{F}_2[\bm{x}]$ of degree at most $d=|\bm{a}|$ and choosing linearly independent ones in a greedy manner).
By Lemma~\ref{lem:basis}, $\mathcal{B}$ is also a basis for $A$. 
Let $\{m_1,\ldots,m_k\}=\mathcal{B} \cap \widetilde{\mathbbm{k}}[\bm{x}]_{\bm{b}}$ be a basis of $A_{\bm{b}}$.
Then $\times\bm{\ell}^{\bm{a}-2\bm{b}}:A_{\bm{b}} \rightarrow A_{\bm{a}-\bm{b}}$ is injective if and only if the set of polynomials $\mathcal{S}=\{m_l\bm{\ell}^{\bm{a}-2\bm{b}} \bm{\theta}^{\bm{\alpha}}: l=1,\ldots,k, \bm{\alpha} \in \mathbb{N}^d\}$ is linearly independent in $\widetilde{\mathbbm{k}}[\Delta]$.
By Theorem~\ref{thm:main multiHL} for a field of characteristic $2$, we know that $\mathcal{S}$ is linearly independent in $\widetilde{\mathbb{F}_2}[\Delta]$. 
By the similar argument as in the proof of Lemma~\ref{lem:basis}, this implies the linear independence of $\mathcal{S}$ in $\widetilde{\mathbbm{k}}[\Delta]$.
Thus the map $\times\bm{\ell}^{\bm{a}-2\bm{b}}:A_{\bm{b}} \rightarrow A_{\bm{a}-\bm{b}}$ is injective.
\end{proof}

\section{Manifolds, simplicial cycles, 2-CM complexes}
\subsection{Manifolds}
For homology manifolds over $\mathbb{F}_2$, we have the following theorem, which can be seen as a multigraded version of almost strong Lefschetz property.
\begin{theorem} \label{thm:manifold almost HL}
For $\bm{a} \in \mathbb{N}^m_+$, let $(\Delta,\kappa)$ be an $\bm{a}$-balanced homology manifold over $\mathbb{F}_2$.
For a field $\mathbbm{k}$ of characteristic $0$ or $2$, let $A=\widetilde{\mathbbm{k}}[\Delta]/(\Theta)$ be the generic $\mathbb{N}^m$-graded Artinian reduction of $\mathbbm{k}[\Delta]$, and let $\ell_j$ be a generic element in $A_{\bm{e}_j}$ for $j=1,\ldots,m$.
Then, for any $\bm{b} \in \mathbb{N}^m$ and $j \in [m]$ with $\bm{b} \leq \frac{\bm{a}-\bm{e}_j}{2}$, the multiplication map
\[
\times \bm{\ell}^{\bm{a}-2\bm{b}-\bm{e}_j} : A_{\bm{b}+\bm{e}_j} \to A_{\bm{a}-\bm{b}}
\]
is surjective.
\end{theorem}
To prove Theorem~\ref{thm:manifold almost HL}, we highlight useful conventions for star and link.
Let $\Theta$ be an l.s.o.p.~for $\mathbbm{k}[\Delta]$.
Recall that $\Theta$ is identified with a point configuration $p:V(\Delta) \rightarrow \mathbbm{k}^{d}$, where $\dim \Delta +1$.
Let $\tau \in \Delta$ be a face.
The l.s.o.p.~for the star is obtained by the restriction of $p$ to $V(\st_\tau \Delta)$, and an l.s.o.p.~for the link is obtained by the projection of $p|_{V(\st_\tau \Delta)}$ to $\mathbbm{k}^{d}/\spa p(\tau)$\footnote{To obtain an l.s.o.p.~for the link explicitly, one needs to identify $\mathbbm{k}^{d}/\spa p(\tau)$ with $\mathbbm{k}^{d-|\tau|}$. Here, up to isomorphism, the resulting Artinian reduction of $\mathbbm{k}[\lk_\tau \Delta]$ does not depend on the identification.}.
Throughout, we always assume this convention for star and link, and when $A=\mathbbm{k}[\Delta]/(\Theta)$ is the Artinian reduction of $\mathbbm{k}[\Delta]$ with respect to $\Theta$, the corresponding Artinian reduction of $\mathbbm{k}[\st_\tau \Delta]$ and $\mathbbm{k}[\lk_\tau \Delta]$ are denoted by $A(\st_\tau \Delta)$ and $A(\lk_\tau \Delta)$, respectively.
Note also that if $(\Delta,\kappa)$ is $\bm{a}$-balanced and $\Theta$ is an $\mathbb{N}^m$-graded l.s.o.p., for a face $\tau \in \Delta^{(\bm{b})}$, $\lk_\tau \Delta$ is $(\bm{a}-\bm{b})$-balanced by the restriction of $\kappa$ and the corresponding l.s.o.p.~for $\mathbbm{k}[\lk_\tau \Delta]$ is also $\mathbb{N}^m$-graded.
Here, for an $\bm{a}$-balanced simplicial complex $(\Delta,\kappa)$, we denote $\Delta^{(\bm{b})}=\{\tau\in\Delta:|\tau \cap \kappa^{-1}(j)|=b_j \text{ for all }j\in [m]\}$.
We have the following lemmas.
\begin{lemma}[Cone lemma] \label{lem:cone}
Let $(\Delta,\kappa)$ be an $\bm{a}$-balanced simplicial complex with $\bm{a}\in\mathbb{N}^m$ and let $\Theta$ be an $\mathbb{N}^m$-graded l.s.o.p.~for $\mathbbm{k}[\Delta]$.
Let $A=\mathbbm{k}[\Delta]/(\Theta)$ be the Artinian reduction of $\mathbbm{k}[\Delta]$ with respect to $\Theta$.
Then for any vertex $v \in V(\Delta)$, there is a degree preserving isomorphism of $\mathbb{N}^m$-graded algebras
\[
A(\lk_v \Delta)_{*} \cong A(\st_v \Delta)_{*}.
\]
\end{lemma}
\begin{proof}
For completeness, we include the proof of this well-known lemma \footnote{See also~\cite[Theorem 7]{Lee} for a vector space isomorphism version of the statement (over $\mathbb{R}$).}.
Let $d=\dim \Delta+1$ and $R=\mathbbm{k}[x_v:v \in V(\st_v\Delta)]$.
To reflect our convention for the star, let us denote the l.s.o.p.~for $\mathbbm{k}[\st_v \Delta]$ also as $\theta_1,\ldots,\theta_d$.
By our convention for the link, for the l.s.o.p.~$\theta_1', \ldots, \theta_{d-1}'$ for $\mathbbm{k}[\lk_v \Delta]$, we have an identity of ideals $(x_v,\theta_1',\ldots,\theta_{d-1}')=(\theta_1,\ldots,\theta_d)$ in $R$.
Thus, we have a desired isomorphism since $I_{\lk_v \Delta} = (x_v) + I_{\st_v \Delta}$.
\end{proof}
\begin{lemma} \label{lem:partition}
Let $(\Delta,\kappa)$ be an $\bm{a}$-balanced simplicial complex with $\bm{a}\in\mathbb{N}^m$ and let $\Theta$ be an $\mathbb{N}^m$-graded l.s.o.p.~for $\mathbbm{k}[\Delta]$.
Let $A=\mathbbm{k}[\Delta]/(\Theta)$ be the Artinian reduction of $\mathbbm{k}[\Delta]$ with respect to $\Theta$.
Then, for each $j \in [m]$, there is a degree preserving surjection 
\[
\bigoplus_{v \in \Delta^{(\bm{e}_j)}} A(\st_v \Delta)_{*} \twoheadrightarrow A(\Delta)_{*+\bm{e}_j}.
\]
\end{lemma}
\begin{proof}
For each $v \in \Delta^{(\bm{e}_j)}$, the multiplication by $x_v$ induces a map $\varphi_v: \mathbbm{k}[\st_v \Delta]_* \overset{\cdot x_v}{\rightarrow} \mathbbm{k}[\Delta]_{*+\bm{e}_j}$. Consider their sum
\[
\varphi:\bigoplus_{v \in \Delta^{(\bm{e}_j)}} \mathbbm{k}[\st_v \Delta]_* \rightarrow \mathbbm{k}[\Delta]_{*+\bm{e}_j}
\]
over all $v \in \Delta^{(\bm{e}_j)}$.
Then $\varphi$ is surjective since every monomial of $\mathbbm{k}[\bm{x}]$ with degree at least $\bm{e}_j$ is divisible by some $x_v$ with $v \in \Delta^{(\bm{e}_j)}$.
So, $\varphi$ induces a surjection between the Artinian reductions.
\end{proof}
\begin{proof}[Proof of Theorem~\ref{thm:manifold almost HL}]
We have the following commuting diagram:
\[
\begin{tikzpicture}[auto]
\node (a) at (0,1.2) {$A(\Delta)_{\bm{b}+\bm{e}_j}$}; 
\node (b) at (6,1.2) {$A(\Delta)_{\bm{a}-\bm{b}}$};
\node (c) at (0,0) {$\bigoplus_{v \in \Delta^{(\bm{e}_j)}} A(\lk_v \Delta)_{\bm{b}}$}; 
\node (d) at (6,0) {$\bigoplus_{v \in \Delta^{(\bm{e}_j)}} A(\lk_v \Delta)_{\bm{a}-\bm{b}-\bm{e}_j}$};
\draw[->] (a) to node {$\times \bm{\ell}^{\bm{a}-2\bm{b}-\bm{e}_j}$} (b);
\draw[->] (c) to node[above] {$\times \bm{\ell}^{\bm{a}-2\bm{b}-\bm{e}_j}$} node[below] {$\cong$} (d);
\draw[->>] (c) to (a);
\draw[->>] (d) to (b);
\end{tikzpicture}
\]
Here, the vertical maps are the composite of an isomorphism in Lemma~\ref{lem:cone} and a surjection in Lemma~\ref{lem:partition}.
For each $v \in \Delta^{(\bm{e}_j)}$, $(\lk_v \Delta,\kappa_{|V(\lk_\tau \Delta)})$ is an $(\bm{a}-\bm{e}_j)$-balanaced homology sphere over $\mathbb{F}_2$, and the l.s.o.p.~for the link and the restriction of $\ell_j$s to the link is generic enough as the given $\Theta$ and $\ell_j$s are generic. 
Thus by Theorem~\ref{thm:main multiHL}, the map
\[
\times \bm{\ell}^{\bm{a}-2\bm{b}-\bm{e}_j} :A(\lk_v \Delta)_{\bm{b}} \rightarrow A(\lk_v \Delta)_{\bm{a}-\bm{b}-\bm{e}_j}
\]
is an isomorphism.
The horizontal map below in the diagram is a composite of these isomorphisms, so it is an isomorphism. Thus the horizontal map above is surjective.
\end{proof}

Let us derive a numerical consequence Theorem~\ref{thm:manifold flag_h} of Theorem~\ref{thm:manifold almost HL} in terms of flag $h''$-vector, a suitable modification of flag $h$-vector.
Given a field $\mathbbm{k}$, for an $\bm{a}$-balanced simplicial complex $(\Delta,\kappa)$, the \emph{flag $h'$-vector} $(h'_{\bm{b}})_{\bm{0} \leq \bm{b} \leq \bm{a}}$ and the \emph{flag $h''$-vector} $(h''_{\bm{b}})_{\bm{0} \leq \bm{b} \leq \bm{a}}$ is defined by 
\begin{align*}
h'_{\bm{b}}&=h_{\bm{b}} -\binom{\bm{a}}{\bm{b}} \left( \sum_{j=1}^{|\bm{b}|-1}(-1)^{|\bm{b}|-j}\widetilde{\beta}_{j-1} \right), \\
h''_{\bm{b}} &= 
    \begin{cases}
    h'_{\bm{b}} - \binom{\bm{a}}{\bm{b}} \widetilde{\beta}_{|\bm{b}|-1} & (\text{if } \bm{b} \neq \bm{a})\\
    h'_{\bm{b}} &(\text{if } \bm{b} = \bm{a})
    \end{cases}.
\end{align*}
Note that for Cohen-Macaulay complex, both the flag $h'$-, $h''$-vector coincide with the flag $h$-vector.
For every $(d-1)$-dimensional simplicial complex, its $h'$-vector $(h'_i)_{0 \leq i \leq d}$ and $h''$-vector $(h''_i)_{0 \leq i \leq d}$ is defined as the flag $h'$-, $h''$-vector considered as a monochromatic simplicial complex.
Note that $h'_i=\sum_{|\bm{b}|=i} h'_{\bm{b}}$ and $h''_i=\sum_{|\bm{b}|=i} h''_{\bm{b}}$ hold.

For homology manifolds, algebraic interpretations of $h'$-, $h''$-vectors are given by Schenzel~\cite{Sch81} and Novik and Swartz~\cite{NS09}.
The following lemma is the balanced analogue of these interpretations about flag $h'$-, $h''$-vectors.
A homology $(d-1)$-manifold $\Delta$ over $\mathbbm{k}$ is \emph{orientable} if $\widetilde{\beta}_{d-1}$ is equal to the number of connected components of $\Delta$.
Recall that, for a graded $\mathbbm{k}[x_1,\ldots,x_n]$-module $M$, its socle is the submodule $\Soc(M)=\{a \in M: \mathfrak{m}a=0\}$, where $\mathfrak{m}=(x_1,\ldots,x_n)$ is the maximal graded ideal.
\begin{lemma} \label{lem:flag h', h''}
Let $(\Delta,\kappa)$ be an $\bm{a}$-balanced connected homology manifold over $\mathbbm{k}$ and let $\Theta$ be an $\mathbb{N}^m$-graded l.s.o.p.~for $\mathbbm{k}[\Delta]$.
Let $A=\mathbbm{k}[\Delta]/(\Theta)$ be an Artinian reduction of $\mathbbm{k}[\Delta]$ with respect to $\Theta$. Then,
\begin{enumerate}
\item[(i)] (Schenzel's formula~\cite[Theorem 3.1]{JMNS}) $h'_{\bm{b}}=\dim A_{\bm{b}}$ for each $\bm{b} \in \mathbb{N}^m$.
\item[(ii)] \cite[Corollary 3.3]{JMNS} $h''_{\bm{b}}=\dim A_{\bm{b}}/\Soc^\circ_{\bm{b}}$ for each $\bm{b} \in \mathbb{N}^m$, where $\Soc^\circ=\bigoplus_{\bm{0} \leq \bm{b} \lneq \bm{a}}\Soc(A)_{\bm{b}}$ denotes the internal socle of $A$.
\item[(iii)] (Dehn-Sommerville relation~\cite[Theorem 4.1]{JMNS}) Moreover if $\Delta$ is orientable, $h''_{\bm{b}}=h''_{\bm{a}-\bm{b}}$ for each $\bm{b} \in \mathbb{N}^m$.
\end{enumerate}
\end{lemma}
We remark that every homology manifold over a field of characteristic $2$ is orientable.
Now we are ready to prove Theorem~\ref{thm:manifold flag_h}.
\begin{proof}[Proof of Theorem~\ref{thm:manifold flag_h}]
Let $\bm{b}, \bm{c} \in \mathbb{N}^m$ be integer vectors with $\bm{b} \lneq \bm{c} \lneq \bm{a}-\bm{b}$.
Then there is some $j\in[m]$ with $\bm{b}+\bm{e}_j \leq \bm{c}$.
By Theorem~\ref{thm:manifold almost HL} the composite $A(\Delta)_{\bm{b}+\bm{e}_j} \overset{\times \bm{\ell}^{\bm{c}-\bm{b}-\bm{e}_j}}{\longrightarrow} A(\Delta)_{\bm{c}} \overset{\times \bm{\ell}^{\bm{a}-\bm{b}-\bm{c}}}{\longrightarrow} A(\Delta)_{\bm{a}-\bm{b}}$ is surjective, so the latter map $\times \bm{\ell}^{\bm{a}-\bm{b}-\bm{c}}:A(\Delta)_{\bm{c}} \rightarrow A(\Delta)_{\bm{a}-\bm{b}}$ is surjective.
Since $\bm{c} \lneq \bm{a}-\bm{b}$, the degree $\bm{c}$ component of the socle is contained in the kernel of this map. 
Thus by Lemma~\ref{lem:flag h', h''}, we have the inequality
\begin{equation}
h''_{\bm{c}}=\dim A(\Delta)_{\bm{c}} - \dim \Soc(A(\Delta))_{\bm{c}} \geq \dim A(\Delta)_{\bm{a}-\bm{b}} = h'_{\bm{a}-\bm{b}}. \label{eq:mfd_h}
\end{equation}
All that remains is to convert (\ref{eq:mfd_h}) to the inequality between flag $h''$-vectors.
For this, if $\bm{b}\neq\bm{0}$, by Lemma~\ref{lem:flag h', h''} (iii), we have 
\begin{align}
h'_{\bm{a}-\bm{b}} &=h''_{\bm{a}-\bm{b}}+\binom{\bm{a}}{\bm{b}}\widetilde{\beta}_{|\bm{a}-\bm{b}|-1} & \notag \\
&=h''_{\bm{b}} + \binom{\bm{a}}{\bm{b}}\widetilde{\beta}_{|\bm{b}|} & (\widetilde{\beta}_{k}=\widetilde{\beta}_{d-k-1} \text{ for } k \geq 1 \text{ by Poincar\'{e} duality}), \label{eq:mfd_betti}
\end{align}
while the same equality (\ref{eq:mfd_betti}) also holds for $\bm{b}=\bm{0}$ since $h'_{\bm{a}}=1$ and $h''_{\bm{0}}+\binom{\bm{a}}{\bm{0}}\beta_{0}=1+0=1$. By (\ref{eq:mfd_h}) and (\ref{eq:mfd_betti}), the desired inequality $h''_{\bm{c}} \geq h''_{\bm{b}}+\binom{\bm{a}}{\bm{b}}\widetilde{\beta}_{|\bm{b}|}$ follows.
\end{proof}

\begin{corollary} \label{cor:mfd-h}
For $\bm{a} \in \mathbb{N}^m_+$, let $(\Delta,\kappa)$ be an $\bm{a}$-balanced connected homology manifold over $\mathbb{F}_2$. We have
\[
\frac{h_i''+\binom{d}{i}\widetilde{\beta}_{i}}{\binom{m+i-1}{i}} \leq \frac{h_{i+1}''}{\binom{m+i}{i+1}}
\]
for every nonnegative integer $i \in \mathbb{N}$ with $i\leq \min_{j=1}^m \frac{a_j-1}{2}$, where $d=|\bm{a}|$.
\end{corollary}
\begin{proof}
The inequality between $h''$-vectors is obtained by taking a weighted sum of the inequalities $h''_{\bm{b}} + \binom{\bm{a}}{\bm{b}}\widetilde{\beta}_{|\bm{b}|} \leq h''_{\bm{b}+\bm{e}_j}$ over all $\bm{b} \in \mathbb{N}^m$ with $|\bm{b}|=i$ and $j \in [m]$ as follows:
\[
(m+i) (h''_i+\binom{d}{i}\widetilde{\beta}_i)
=\sum_{|\bm{b}|=i}\sum_{j=1}^m (b_j+1)\left(h''_{\bm{b}}+\binom{\bm{a}}{\bm{b}}\widetilde{\beta}_i \right) \leq \sum_{|\bm{b}|=i}\sum_{j=1}^m (b_j+1)h''_{\bm{b}+\bm{e}_j}
=(i+1) h''_{i+1}.
\]
\end{proof}

\subsection{Simplicial Cycle}
For a $(d-1)$-dimensional simplicial complex $\Delta$ with $\widetilde{H}_{d-1}(\Delta;\mathbbm{k}) \neq 0$, we call a nonzero element $\mu \in \widetilde{H}_{d-1}(\Delta;\mathbbm{k})$ a \emph{simplicial cycle}.
Let $\Delta$ be a $(d-1)$-dimensional simplicial complex, and let $A=\mathbbm{k}[\Delta]/(\Theta)$ be an Artinian reduction of $\mathbbm{k}[\Delta]$.
A simplicial cycle $\mu$ of $\Delta$ defines a nonzero linear function $\Psi_\mu:A_{d} \rightarrow \mathbbm{k}$ by 
\begin{equation}
\Psi_\mu(x_\sigma) = \frac{\mu_\sigma}{[\sigma]} \label{eq:cycle}
\end{equation}
for each positively facet $\sigma \in \Delta$, where $\mu_\sigma$ is the coefficient of the facet $\sigma$ and $[\sigma]$ is computed as in Lemma~\ref{lem:Lee formula} (note that the linear function on $A_d$ is determined by the values of squarefree monomials~\cite[Theorem 9]{Lee}).
Here the fact that there exists a linear map that satisfies (\ref{eq:cycle}) for all facets of $\Delta$ can be verified by checking equilibrium condition~\cite[Theorem 10]{Lee}. (Or alternatively consider weighted connected sum in the formulation of Karu-Xiao~\cite{KX}.)

A nonzero linear function $\varphi:\mathbbm{k}[\bm{x}]_d \rightarrow \mathbbm{k}$ determines a standard Artinian Gorenstein graded algebra $\mathbbm{k}[\bm{x}]/I$ by the condition that $f \in \mathbbm{k}[\bm{x}]_k$ is in $I$ if and only if either $k \geq d+1$ or $\varphi(fg)=0$ for all $g \in \mathbbm{k}[\bm{x}]_{d-k}$.
Thus the map $\Psi_\mu$ (by concatenating a projection $\mathbbm{k}[\Delta]_{d} \twoheadrightarrow A_{d}$ to it) determines an Artinian Gorenstein $\mathbbm{k}$-algebra denoted as $B(\mu)$, which is called the \emph{Gorensteinification} of $A$~\cite{APP}.
Note that if $\Delta$ is a homology sphere, $A$ itself is Gorenstein, so we have $A=B(\mu)$ for nonzero $\mu \in \widetilde{H}_{d-1}(\Delta;\mathbbm{k})$.
If $A$ is $\mathbb{N}^m$-graded, $B(\mu)$ is also $\mathbb{N}^m$-graded and satisfies multigraded Poincar\'{e} duality.

Theorem~\ref{thm:main multiHL} is generalized to a simplicial cycle over a field of characteristic $2$ as follows.
\begin{theorem} \label{thm:simplicial cycle}
For an $\bm{a}$-balanced simplicial complex $(\Delta,\kappa)$, let $\mu$ be a simplicial cycle of $\Delta$ over a field $\mathbbm{k}$ of characteristic $2$.
Let $A=\widetilde{\mathbbm{k}}[\Delta]/(\Theta)$ be the generic $\mathbb{N}^m$-graded Artinian reduction of $\mathbbm{k}[\Delta]$ and let $B(\mu)$ be the Gorensteinification of $A$ with respect to $\mu$.
Then for generic elements $\ell_j \in B(\mu)_{\bm{e}_j}$ for $j=1,\ldots,m$, the multiplication map
\[
\times \bm{\ell}^{\bm{a}-2\bm{b}}: B(\mu)_{\bm{b}} \rightarrow B(\mu)_{\bm{a}-\bm{b}}
\]
is an isomorphism for every $\bm{b}\in\mathbb{N}^m$ with $\bm{b} \leq \frac{\bm{a}}{2}$.
\end{theorem}
\begin{proof}[Proof Sketch]
We only give a proof sketch as the discussion is almost the same as in Section 3-5.
By the similar discussion as in~\cite[Section 2.6]{KX}, by introducing the cone vertex, the simplicial cycle $\mu$ can be written as a weighted connected sum of simplex boundaries with the weights in $\mathbbm{k}$.
Accordingly, the map $\Psi_\mu$ is written as the weighted sum of the evaluation map of simplex boundaries. 
Hence, $\Psi=\Psi_\mu$ also satisfies the differential formula in Lemma~\ref{lem:diff}, and thus Corollary~\ref{cor:diff general}.
In the argument in Section 5, we have only used the fact that $A$ has multigraded Poincar\'{e} duality, which is now passed to $B(\mu)$. Hence we can reproduce the argument in Section 5 by replacing $A$ with $B(\mu)$.
\end{proof}
\begin{remark}
To deduce a (generalized) lower bound type inequality from Theorem~\ref{thm:simplicial cycle}, it is a key to understand $\dim A_*$ and $\dim B_*-\dim A_*$.
As far as I know, even in the natural $\mathbb{N}$-graded case, we have limited knowledge of these numbers other beyond the case of homology manifolds. 
(\cite{Mu,Fog,Oba} can be though of as a research in this direction.)
The understanding of these numbers would also be important from the perspective of skeletal rigidity~\cite{TW}.
\end{remark}
From Theorem~\ref{thm:simplicial cycle}, we can obtain the following multigraded version of top-heavy Lefschetz property for doubly Cohen-Macaulay complexes.
A simplicial complex $\Delta$ is called a \emph{doubly Cohen-Macaulay} complex (or \emph{$2$-CM} complex) over a field $\mathbbm{k}$ if $\Delta$ is Cohen-Macaulay over $\mathbbm{k}$ and for each vertex $v$, $\Delta-v$ is Cohen-Macaulay over $\mathbbm{k}$ and has the same dimension as $\Delta$.
\begin{theorem} \label{thm:2CM lefschetz}
For $\bm{a} \in \mathbb{N}^m_+$, let $(\Delta,\kappa)$ be an $\bm{a}$-balanced doubly Cohen-Macaulay complex over a field $\mathbbm{k}$ of characteristic $2$.
Let $A=\widetilde{\mathbbm{k}}[\Delta]/(\Theta)$ be the generic $\mathbb{N}^m$-graded Artinian reduction of $\mathbbm{k}[\Delta]$.
Then, for generic elements $\ell_j \in A_{\bm{e}_j}$ for $j=1,\ldots,m$, the multiplication map
\[
\times \bm{\ell}^{\bm{a}-2\bm{b}}: A_{\bm{b}} \rightarrow A_{\bm{a}-\bm{b}}
\]
is injective for any $\bm{b} \in \mathbb{N}^m$ with $\bm{b} \leq \frac{\bm{a}}{2}$.
\end{theorem}
\begin{proof}
Let $d=|\bm{a}|$ and let $\mu_1,\ldots,\mu_k$ be a basis of $\widetilde{H}_{d-1}(\Delta;\mathbbm{k})$.
It is known that an Artinian reduction of the Stanley-Reisner ring of a doubly Cohen-Macaulay complex is a level ring~\cite[Section III.3]{Sta-book}, that is, $\Soc (A)=A_{\bm{a}}$.
This implies that for any nonzero $x \in A_{\bm{b}}$, there is $y\in A_{\bm{a}-\bm{b}}$ such that $(\Psi_{\mu_1}(xy),\ldots,\Psi_{\mu_k}(xy))\neq0$.
Thus the map $A(\Delta)_* \rightarrow \bigoplus_{i=1}^k B(\mu_i)_*$ is injective.
Consider the following commuting diagram:
\[
\begin{tikzpicture}[auto]
\node (a) at (0,1.2) {$A(\Delta)_{\bm{b}}$}; 
\node (b) at (4,1.2) {$A(\Delta)_{\bm{a}-\bm{b}}$};
\node (c) at (0,0) {$\bigoplus_{l=1}^k B(\mu_l)_{\bm{b}}$}; 
\node (d) at (4,0) {$\bigoplus_{l=1}^k B(\mu_l)_{\bm{a}-\bm{b}}$};
\draw[->] (a) to node {$\times \bm{\ell}^{\bm{a}-2\bm{b}}$} (b);
\draw[->] (c) to node[above] {$\times \bm{\ell}^{\bm{a}-2\bm{b}}$} node[below] {$\cong$} (d);
\draw[right hook->] (a) to (c);
\draw[right hook->] (b) to (d);
\end{tikzpicture}
\]
As the bottom map is an isomorphism by Theorem~\ref{thm:simplicial cycle}, the top map is injective.
\end{proof}
Now Theorem~\ref{thm:2CM_flag h} is readily derived.
\begin{proof}[Proof of Theorem~\ref{thm:2CM_flag h}]
For $\bm{b} \leq \bm{c} \leq \bm{a}-\bm{b}$, Theorem~\ref{thm:2CM lefschetz} implies that the linear map $\times \bm{\ell}^{\bm{b}-\bm{c}}:A_{\bm{b}} \rightarrow A_{\bm{c}}$ is injective. Thus $h_{\bm{b}}=\dim A_{\bm{b}} \leq \dim A_{\bm{c}} = h_{\bm{c}}$ follows.
\end{proof}
We remark that the inequality between the $h$-vectors as in Corollary~\ref{cor:h-vec} also follows by the same argument.
\section{Lefschetz property as an $\mathbb{N}$-graded algebra}
Let $(\Delta,\kappa)$ be an $\bm{a}$-balanced simplicial complex for $\bm{a}\in\mathbb{N}_+^m$ with $d=|\bm{a}|$.
For simplicity, in this section, we focus on the case that $\Delta$ is a homology sphere.
Let $A$ be the generic $\mathbb{N}^m$-graded Artinian reduction of $\mathbbm{k}[\Delta]$.
In this section, instead of considering $A$ as an $\mathbb{N}^m$-graded algebra, we regard $A$ as an $\mathbb{N}$-graded algebra $A=A_0\oplus\cdots\oplus A_d$ under the coarse grading $\deg x_v=1$ for all $v \in V(\Delta)$ and investigate weak Lefschetz property as an $\mathbb{N}$-graded algebra.
See \cite[Conjecture 1.3]{CJMN} and \cite[Conjecture 1.1]{Oba} for related conjectures.

\subsection{Full-rankness at ends}
From Theorem~\ref{thm:main multiHL}, one can deduce that, for a generic linear form $\ell \in A_1$, the multiplication map $\times \ell:A_i \rightarrow A_{i+1}$ is injective for $i \leq \min_{j=1}^m \frac{a_j-1}{2}$ and is surjective for $i \geq d-1-\min_{j=1}^m \frac{a_j-1}{2}$.
We can slightly increase the range of $i$ as follows. 
\begin{theorem} \label{thm:fullrank at ends}
For $\bm{a} \in \mathbb{N}^m_+$ with $|\bm{a}|=d$, let $(\Delta,\kappa)$ be an $\bm{a}$-balanced homology $(d-1)$-sphere over $\mathbb{F}_2$ and let $\mathbbm{k}$ be a field of characteristic $0$ or $2$.
Let $A=\widetilde{\mathbbm{k}}[\Delta]/(\Theta)$ be the generic $\mathbb{N}^m$-graded Artinian reduction of $\mathbbm{k}[\Delta]$, and let $\ell=\sum_{v \in V(\Delta)} x_v \in A_1$.
Then the multiplication map $\times \ell:A_i \rightarrow A_{i+1}$ is injective for $i \leq \min \{ \frac{d-1}{2}, \frac{a_1}{2},\ldots,\frac{a_m}{2} \}$ and is surjective for $i \geq d-1-\min\{ \frac{d-1}{2}, \frac{a_1}{2} ,\ldots, \frac{a_m}{2} \}$.
\end{theorem}
\begin{proof}
As the characteristic $0$ case follows from characteristic $2$ case by the similar argument as in Section 6, we assume that $\mathbbm{k}$ is a field of characteristic $2$.
We first show the injectivity for $i \leq \min \{ \frac{d-1}{2}, \frac{a_1}{2},\ldots,\frac{a_m}{2} \}$.
For this, it suffices to show that the quadratic form $\mathcal{Q}:A_i \rightarrow A_{2i+1}$ defined by $\mathcal{Q}(g)=g^2 \ell$ satisfies the property that $\mathcal{Q}(g) \neq 0$ if $g \neq 0$.
Suppose that $g$ is a nonzero element in $A_i$.
Then by Poincar\'{e} duality (as an $\mathbb{N}$-graded algebra), there is a monomial $x_K$ of total degree $d-i$ such that $gx_K \neq 0$ in $A_d$.
Let $\bm{b}$ be the degree of the monomial $x_K$ in $\mathbb{N}^m$-grading.
Then we have $\bm{0} \leq \bm{b} \leq \bm{a}$ and $|\bm{b}|=d-i$. 
By $i \leq \min_{j \in [m]} \frac{a_j}{2}$, we have $\bm{b} \geq \frac{\bm{a}}{2}$. 
By this and $i \leq \frac{d-1}{2}$, the square $x_K^2$ can be written as $x_K^2=x_Ix_{v^*}x_J$ for some $\kappa$-transversal sequences $I$, some vertex $v^* \in V(\Delta)$, and some length $d-2i-1$ sequence $J$ of vertices.
We have the identity
\begin{align*}
\partial_I\Psi(\mathcal{Q}(g)x_J) = \sum_{v \in V(\Delta)} \partial_I\Psi(g^2x_vx_J)
&=\sum_{v \in V(\Delta)} \Psi(g\sqrt{x_Ix_vx_J})^2 & (\text{by Corollary~\ref{cor:diff general}})\\
&\overset{(*)}{=}\Psi(g\sqrt{x_Ix_{v^*}x_J})^2 = \Psi(gx_K)^2 \neq 0,
\end{align*}
where $(*)$ follows from the uniqueness of a variable $x_u$ with $\sqrt{x_Ix_ux_J} \neq 0$ for fixed $I$ and $J$. Thus, we have $\mathcal{Q}(g) \neq 0$. Hence the desired injectivity is derived.

For the remaining surjectivity, note that, by Gorensteiness, $A_{d-i'}$ is a dual vector space of $A_{i'}$ for each $i'$.
Thus, the surjectivity of $\times \ell:A_{d-1-i} \rightarrow A_{d-i}$ is equivalent to the injectivity of $\times \ell :A_{i} \rightarrow A_{i+1}$.
Hence the desired surjectivity at the other end follows.
\end{proof}
\begin{remark}
    For example, if $(a_1,\ldots,a_m)=(2,\ldots,2)$, 
\end{remark}
\subsection{Examples that fail the $\mathbb{N}$-graded weak Lefschetz property}
In Theorem~\ref{thm:fullrank at ends}, we have shown that $\times\ell:A_i \rightarrow A_{i+1}$ is generically full-rank at appropriate ends.
In contrast to~Theorem~\ref{thm:fullrank at ends}, we construct examples of an $\bm{a}$-balanced simplicial sphere such that the full-rankness generically fails around the middle degree.
In~\cite{CJMN,Oba}, such an example is already obtained for $\bm{a}=(a_1,1)$ with $a_1 \geq 2$. Let us start by recalling this construction.

The \emph{stellar subdivision} of $\Delta$ at a face $\sigma \in \Delta$ is the simplicial complex 
\begin{equation*}
(\Delta -\sigma) \cup \{\{v_\sigma\} \cup \tau: \tau \in \st_{\sigma}(\Delta), \tau \not\supset \sigma \},
\end{equation*} 
where $v_\sigma$ is a new vertex.
Note that stellar subdivision preserves the underlying topological space.
Consider a stacked $a_1$-sphere, and apply stellar subdivisions at every facet of it. 
The resulting simplicial sphere $\Delta$ is stacked and $(a_1,1)$-balanced with the $2$-coloring $\kappa$ defined by $\kappa(v)=2$ if and only if $v$ is a new vertex.
It is shown in \cite{CJMN,Oba} that for any choice of $\mathbb{N}^2$-graded l.s.o.p.~$\Theta$ of $\mathbbm{k}[\Delta]$, in the Artinian reduction $A=\mathbbm{k}[\Delta]/(\Theta)$, the map $\times\ell:A_1 \rightarrow A_2$ is degenerate for any $\ell \in A_1$.

To generalize this construction, we consider a partial barycentric subdivision~\cite{AW}.
We denote by $\Delta^{(i)}$ (resp.~$\Delta^{(\leq i)}$, $\Delta^{(\geq i)}$) the set of all faces of $\Delta$ of dimension equal to (resp.~at most, at least) $i$.
For a subset $\mathcal{S}$ of the power set of a finite set, the simplicial complex \emph{spanned by $\mathcal{S}$} is $\{T:T\subset S \text{ for some } S \in \mathcal{S}\}$.
For $0 \leq l < d$ and a pure $(d-1)$-dimensional simplicial complex $\Delta$, the \emph{$l$-th partial barycentric subdivision} $\sd^l(\Delta)$ of $\Delta$ is defined as follows: For each $\tau \in \Delta^{(\geq d-l)}$, let $v_{\tau}$ be a new vertex associated to $\tau$, and define $\sd^l(\Delta)$ as the simplicial complex spanned by the sets $\tau_0 \cup \{v_{\tau_1}, \ldots, v_{\tau_l}\}$ over all flags of faces $\tau_0 \subsetneq \tau_1 \subsetneq \cdots \subsetneq \tau_l$ of $\Delta$ with $\dim \tau_i=d-l-1+i$ for $i=0,\ldots,l$.
Note that we have $\sd^0(\Delta)=\Delta$ and $\sd^{d-1}(\Delta)$ coincides with the barycentric subdivision of $\Delta$. 
Equivalently, the $l$-th partial barycentric subdivision is obtained by ordering the faces of $\Delta^{(\geq d-l)}$ in decreasing order of dimension and then applying stellar subdivisions one by one.
As stellar subdivision preserves the underlying topological space, $\Delta$ and $\sd^l(\Delta)$ have the homeomorphic geometric realizations.
For a pure $(d-1)$-dimensional simplicial complex $\Delta$ and $l<d$, $\sd^l(\Delta)$ is $(d-l,\bm{1}_l)$-balanced with the coloring $\kappa:V(\sd^l(\Delta)) \rightarrow [l+1]$ defined by $\kappa(v)=1$ for $v \in V(\Delta)$ and $\kappa(v_\tau)=|\tau|-d+l+1$ for $v_\tau \in V(\sd^l(\Delta)) \setminus V(\Delta)$, where $\bm{1}_{l}$ is the all ones vector of length $l$.
By grouping the last $l$ colors of $\kappa$ into one color, one can consider $\sd^l(\Delta)$ as a $(d-l,l)$-balanced simplicial complex. We denote the resulting $2$-coloring by $\kappa^\circ$.

Our example is an $i$-th partial barycentric subdivision of a sphere with $h_i=h_{i+1}$. (Such a sphere is called $i$-stacked~\cite{MN}, and a stacked $(d-1)$-sphere is always $i$-stacked for $i < \lfloor d/2\rfloor $.)
More precisely, we have the following. 
\begin{theorem} \label{thm:counterexample}
Let $\mathbbm{k}$ be an infinite field and $i, d$ be positive integers with $i < \frac{d}{2}$.
Let $\Delta$ be a simplicial $(d-1)$-sphere with $h_i(\Delta)=h_{i+1}(\Delta)$, and $\sd^i(\Delta)$ be the $i$-th partial barycentric subdivision of $\Delta$ with the associated $2$-coloring $\kappa^\circ$.
Consider an l.s.o.p.~$\Theta=(\theta_1,\ldots,\theta_d)$ for $\mathbbm{k}[\sd^i(\Delta)]$ such that the $j$-th linear form $\theta_{j}$ is of degree $\bm{e}_2$ under the $\mathbb{N}^2$-grading induced by $\kappa^\circ$ for $j=d-i+1,\ldots,d$.
Then, for the Artinian reduction $A=\mathbbm{k}[\sd^i(\Delta)]/(\Theta)$ with respect to $\Theta$, the multiplication map $\times \ell:A_i \rightarrow A_{i+1}$ is degenerate for any linear form $\ell \in A_1$.

In particular, if $(\sd^i(\Delta),\kappa^\circ)$ is viewed as an $\bm{a}=(d-i,i)$-balanced simplicial complex, there is no $\mathbb{N}^2$-graded l.s.o.p.~$\Theta$ such that the Artinian reduction $\mathbbm{k}[\sd^i(\Delta)]/(\Theta)$ has the weak Lefschetz property as an $\mathbb{N}$-graded algebra.
\end{theorem}
\begin{proof}
As the geometric realizations of $\Delta$ and $\sd^i(\Delta)$ are homeomorphic, $\sd^i(\Delta)$ is a simplicial $(d-1)$-sphere.
As $\sd^i(\Delta)$ is obtained by a sequence of stellar subdivisions of faces of codimension at most $i$, we have $h_{i+1}(\sd^i(\Delta))-h_{i}(\sd^i(\Delta))=h_{i+1}(\Delta)-h_{i}(\Delta)=0$.
Hence $\dim A_i=h_{i}(\sd^i(\Delta))=h_{i+1}(\sd^i(\Delta))=\dim A_{i+1}$ holds. 
Thus it suffices to prove that the multiplication map $\times \ell:A_i \rightarrow A_{i+1}$ cannot be surjective for any linear form $\ell \in A_1$.
Intuitively speaking, in the perspective of skeletal rigidity~\cite{TW},
under the normalization $\ell=\sum_{v \in V(\sd^i(\Delta))} x_v$, the surjectivity of $\times \ell:A_i \rightarrow A_{i+1}$ is equivalent to the affine $i$-stress-freeness of the framework $(\sd^{i}(\Delta),p)$, where $p$ is the point configuration associated to $\Theta$. 
However the $(d-i)$-dimensional subframework induced by $V(\Delta)$ has the same $i$-skeleton as a $(d-1)$-sphere $\Delta$, so it must support an affine $i$-stress in $d-i$-dimension.
We now turn this geometric intuition to an algebraic proof.

Let $W=V(\Delta)$ be the vertex subset of $\sd^i(\Delta)$.
Consider the induced subcomplex $\sd^i(\Delta)_W=\{\tau \in \sd^i(\Delta):\tau \subset W\}$.
By definition, we have $\sd^i(\Delta)_W = \Delta^{(\leq d-i-1)}$.
Let $\pi:\mathbbm{k}[x_v : v \in V(\sd^i(\Delta))] \twoheadrightarrow \mathbbm{k}[x_v : v \in V(\Delta)]$ be the natural projection.
We have $\pi(\mathbbm{k}[\sd^i(\Delta)])=\mathbbm{k}[\sd^i(\Delta)_W]$ by definition and $\pi(\theta_{d-i+1})= \cdots= \pi(\theta_d)=0$ by the assumption. 
Thus for any $\ell \in A_1$ we have
\begin{align} 
&\dim \left(\mathbbm{k}[\sd^i(\Delta)]/(\Theta,\ell)\right)_{i+1} \label{eq:affine stress} \\
\geq& \left(\pi(\mathbbm{k}[\sd^i(\Delta)])/(\pi(\theta_1),\ldots,\pi(\theta_d),\pi(\ell)\right)_{i+1} \notag \\
 =& \left(\mathbbm{k}[\Delta^{(\leq d- i-1)}]/(\widetilde{\theta}_1,\ldots,\widetilde{\theta}_{d-i},\pi(\ell)\right)_{i+1} \notag \\
\overset{(*)}{\geq}& \dim \left(\mathbbm{k}[\Delta^{(\leq d-i-1)}]/(\widetilde{\theta}_1,\ldots,\widetilde{\theta}_{d-i})\right)_{i+1}- \dim \left(\mathbbm{k}[\Delta^{(\leq d- i-1)}]/(\widetilde{\theta}_1,\ldots,\widetilde{\theta}_{d-i})\right)_{i}, \label{eq:skeleton}
\end{align}
where we denote $\widetilde{\theta}_j=\pi(\theta_j)$, and the inequality ($*$) holds since $\pi(\ell)$ is a linear element.
Since the surjectivity of $\times\ell:A_i \rightarrow A_{i+1}$ is equivalent to the value (\ref{eq:affine stress}) being $0$, it suffices to prove that the value (\ref{eq:skeleton}) is positive.

To compute (\ref{eq:skeleton}), note that
$(\widetilde{\theta}_1,\ldots,\widetilde{\theta}_{d-i})$ is an l.s.o.p.~for $\mathbbm{k}[\Delta^{(\leq d-i-1)}]$ by the Kind-Kleinschmidt's criterion~\cite[Lemma III.2.4]{Sta-book}.
Observe also that $\Delta^{(\leq d - i-1)}$ is Cohen-Macaulay as $\Delta$ is Cohen-Macaulay\footnote{For example homological condition of Cohen-Macaulayness can be directly checked as follows. For each $\tau \in \Delta^{(\leq d- i-1)}$, $\lk_\tau(\Delta^{(\leq d- i-1)}) =(\lk_\tau \Delta )^{\leq (d-i-|\tau| -1)}$ holds. So $\widetilde{H}_j(\lk_\tau(\Delta^{(\leq d- i-1)});\mathbbm{k})=\widetilde{H}_j(\lk_\tau \Delta;\mathbbm{k})=0$ for $j \leq d-i-|\tau|-2$.}.
Thus we have
\[
\dim \left(\mathbbm{k}[\Delta^{(\leq d- i-1)}]/(\widetilde{\theta}_1,\ldots,\widetilde{\theta}_{d-i}\right)_{j}=
h_{j}(\Delta^{(\leq d-i-1)}) \qquad \text{ for } j=0,\ldots,d-i.
\]
Now in the Stanley's triangle table for $\Delta$ (see~\cite[p.250]{Zie}), the $h$-vector of $\Delta^{(\leq d-i-1)}$ appears in the $(d-i)$-th row. So the consecutive difference $h_{i+1}(\Delta^{(\leq d-i-1)})-h_{i}(\Delta^{(\leq d-i-1)})$ appears in the $(d-i+1)$-th row. As the $h$-vector of $\sd^i(\Delta)$ is positive, all the entries in the triangle table is positive. So, (\ref{eq:skeleton}) is positive. 
We thus verified the degeneracy of $\times\ell:A_i \rightarrow A_{i+1}$.
\end{proof}

\paragraph{Acknowledgment}
The author thanks Satoshi Murai deeply for asking me motivating questions and also for giving comments on the manuscript. I have also benefited greatly from his unpublished manuscript. 
The author also thanks Shin-ichi Tanigawa for his supervision and helpful discussion.
We thank anonymous referees for careful reading and comments that help improve the presentations.
\bibliographystyle{plain}
\bibliography{myreference}

\end{document}